\documentclass[a4paper,twoside,12pt]{amsart}
\usepackage[mode=image]{standalone} 

\usepackage[margin=4cm]{geometry}
\usepackage[T1]{fontenc}

\usepackage{amsthm, amsmath, amssymb, bbm, bm}
\usepackage{mathrsfs}
\usepackage{yhmath}

\usepackage{accents}

\usepackage{tikz}
\usetikzlibrary{calc}
\usetikzlibrary{intersections}
\usepackage{tikz-3dplot}

\usepackage{enumerate}
\usepackage{graphics}

\usepackage[all]{xy}
\makeatother
\newdir{ >}{{}*!/-10pt/@{>}}

\usepackage[colorlinks=true, pdfstartview=FitV, linkcolor=blue, %
citecolor=blue, urlcolor=blue]{hyperref}

\usepackage{color}

\usepackage{marginnote}
\usepackage{xspace}
\usepackage{ulem}
\usepackage{units}

\usepackage{accents}

\usepackage{tensor}

\makeatletter

\newcommand{\nc}{\newcommand}

\numberwithin{equation}{section}

\theoremstyle{plain}

\newtheorem{theorem}{Theorem}[section]
\newtheorem*{theorem*}{Theorem}

\newtheorem{proposition}[theorem]{Proposition}
\newtheorem{lemma}[theorem]{Lemma}

\newtheorem{sublemma}[theorem]{Sublemma}

\theoremstyle{definition}
\newtheorem{definition}[theorem]{Definition}
\newtheorem{lemmadefinition}[theorem]{Lemma-Definition}

\newtheorem*{example*}{Example}
\newtheorem{notation}[theorem]{Notation}
\newtheorem{remark}[theorem]{Remark}

\nc{\Lemma}{\begin{lemma}}
\nc{\enlemma}{\end{lemma}}
\nc{\Prop}{\begin{proposition}}
\nc{\enprop}{\end{proposition}}
\nc{\Def}{\begin{definition}}
\nc{\edf}{\end{definition}}

\renewcommand{\emptyset}{\varnothing}

\nc{\scup}{\mathop{\scalebox{.8}{$\displaystyle\bigcup$}}\mspace{1mu}\limits}
\nc{\scap}{\mathop{\scalebox{.8}{$\displaystyle\bigcap$}}\limits}
\nc{\ssqcup}{\mathop{\scalebox{.8}{$\displaystyle\bigsqcup$}}\limits}

\newcommand{\dunion}{\sqcup}
\newcommand{\union}{\cup}

\newcommand{\R}{\mathbb{R}}

\DeclareMathOperator{\id}{id}

\newcommand{\derived}[1]{\mathrm{#1}}

\newcommand{\derd}{\derived{D}}
\newcommand{\dere}{\derived{E}}

\newcommand{\derr}{\derived{R}}
\newcommand{\derl}{\derived{L}}
\nc{\derb}{\derd^{\mathrm{b}}}
\newcommand{\BDC}{\derd^{\mathrm{b}}}

\nc{\soplus}{\scalebox{.65}{\raisebox{.2ex}{$\displaystyle\bigoplus$}}}

\newcommand{\dsum}[1][]{\mathbin{\oplus_{#1}}}

\renewcommand{\to}[1][]{\xrightarrow{#1}}
\newcommand{\from}[1][]{\xleftarrow{#1}}

\newcommand{\isoto}[1][]{\xrightarrow[#1]{%
{\raisebox{-.6ex}[0ex][0ex]{$\mspace{1mu}\sim\mspace{2mu}$}}}}

\newcommand{\Endo}[1][]{\mathrm{End}_{\raise1.5ex\hbox to.1em{}#1}}

\newcommand{\Hom}[1][]{\mathrm{Hom}_{\raise1.5ex\hbox to.1em{}#1}}

\newcommand{\RHom}[1][]{\derr\mathrm{Hom}_{\raise1.5ex\hbox to.1em{}#1}}

\newcommand{\Ext}[2][]{\mathrm{Ext}_{\raise1.5ex\hbox to.1em{}#1}^{#2}}

\newcommand{\Tens}[1][]{\mathbin{\otimes_{\raise1.5ex\hbox to-.1em{}#1}}}

\newcommand{\LTens}[1][]{\mathbin{\otimes_{\raise1.5ex\hbox to-.1em{}#1}^{\derl}}}

\newcommand{\Tor}[2][]{\mathrm{Tor}^{\raise1.5ex\hbox to.1em{}#1}_{#2}}

\newcommand{\sheaffont}[1]{\mathcal{#1}}

\def\shi{\sheaffont{I}}

\newcommand{\shendo}[1][]{{\sheaffont{E}nd}_{\raise1.5ex\hbox to.1em{}#1}}

\renewcommand{\hom}[1][]{{\sheaffont{H}om}_{\raise1.5ex\hbox to.1em{}#1}}
\newcommand{\aut}[1][]{{\sheaffont{A}ut}_{\raise1.5ex\hbox to.1em{}#1}}
\newcommand{\inn}[1][]{{\sheaffont{I}nn}_{\raise1.5ex\hbox to.1em{}#1}}

\newcommand{\rhom}[1][]{{\derr\sheaffont{H}om}_{\raise1.5ex\hbox to.1em{}#1}}

\newcommand{\ext}[2][]{{\sheaffont{E}xt}_{\raise1.5ex\hbox to.1em{}#1}^{#2}}

\newcommand{\thom}[1][]{{\sheaffont{T}hom}_{\raise1.5ex\hbox to.1em{}#1}}

\newcommand{\tens}[1][]{\mathbin{\otimes_{\raise1.5ex\hbox to-.1em{}#1}}}

\newcommand{\ltens}[1][]{\mathbin{\otimes_{\raise1.5ex\hbox to-.1em{}#1}^{\derl}}}

\newcommand{\tor}[2][]{{\sheaffont{T}or}^{\raise1.5ex\hbox to.1em{}#1}_{#2}}

\newcommand{\etens}[1][]{\mathbin{\boxtimes_{\raise1.5ex\hbox to-.1em{}#1}}}

\newcommand{\roim}[1]{\derr#1_*}

\newcommand{\reim}[1]{\derr#1_{\mspace{.5mu}!}\mspace{2mu}}

\newcommand{\reeim}[1]{\derr#1_{\mspace{1mu}!!}\mspace{1mu}}

\newcommand{\opb}[1]{#1^{-1}}

\newcommand{\epb}[1]{#1^{\mspace{1.5mu}!}\mspace{2mu}}

\newcommand{\tenstop}[1][]{\mathbin{\hat{\otimes}_{\raise1.5ex\hbox to-.1em{}#1}}}

\newcommand{\homtop}[1][]{\sheaffont{L}_{\raise1.5ex\hbox to.1em{}#1}}

\newcommand{\Homtop}[1][]{\mathrm{L}_{\raise1.5ex\hbox to.1em{}#1}}

\newcommand{\D}{\sheaffont{D}}

\newcommand{\detens}[1][]%
{\mathbin{\boxtimes_{\raise1.5ex\hbox to-.1em{}#1}^{\mspace{2mu}\mathsf{D}}}}

\newcommand{\dtens}[1][]{\mathbin{\otimes_{\raise1.5ex\hbox to-.1em{}#1}^{\mathsf{D}}}}

\renewcommand{\leq}{\leqslant}
\renewcommand{\geq}{\geqslant}

\newcommand{\field}{\mathbf{k}}
\newcommand{\ind}{\mathrm{I}\mspace{2mu}}
\newcommand{\ifield}{\ind\field}

\newcommand{\cl}{\colon}

\newcommand{\ccomp}{\mathbin{\mathop\circ\limits^+}}
\newcommand{\ctens}{\mathbin{\mathop\otimes\limits^+}}
\newcommand{\cetens}[1][]{\mathbin{\mathop\boxtimes\limits^+_{\raise1.5ex\hbox to-.1em{}#1}}}
\newcommand{\cihom}{{\derr\shi hom}^+}

\newcommand{\ihom}[1][]{{\shi hom}_{\raise1.5ex\hbox to.1em{}#1}}
\newcommand{\rihom}[1][]{{\derr\mspace{2mu}\shi hom}_{\raise1.5ex\hbox to.1em{}#1}}
\newcommand{\ii}[1][]{{\sheaffont{I}h}_{\raise1.5ex\hbox to.1em{}#1}}

\newcommand{\indlim}[1][]{\mathop{\text{\rm``$\varinjlim$''}}\limits_{#1}}

\newcommand{\dcomp}[1][]{\mathbin{\circ_{\raise1.5ex\hbox to-.1em{}#1}^{\mathsf{D}}}}

\newcommand{\enh}{\derived{E}}

\newcommand{\BEC}[2][\ifield]{\dere^{\mathrm{b}}(#1_{#2})}

\newcommand{\BECcon}[2][\ifield]{\dere^{\mathrm{b}}_{\bGmp}(#1_{#2})}
\newcommand{\BECp}[2][\ifield]{\dere^{\mathrm{b}}_+(#1_{#2})}

\newcommand{\BECm}[2][\ifield]{\dere^{\mathrm{b}}_-(#1_{#2})}
\newcommand{\BECpm}[2][\ifield]{\dere^{\mathrm{b}}_\pm(#1_{#2})}

\newcommand{\BECast}[2][\ifield]{\dere^{\mathrm{b}}_\ast(#1_{#2})}

\newcommand{\Edual}{\dual^\enh}
 \newcommand{\Eoim}[1]{\enh#1_*}

\newcommand{\Eeeim}[1]{\enh#1_{!!}}

\newcommand{\Eopb}[1]{\enh#1^{-1}}
\newcommand{\Eepb}[1]{\enh\mspace{1mu}#1^{\mspace{1.5mu}!}}

\newcommand{\semicolon}{\nobreak \mskip2mu\mathpunct{}\nonscript\mkern-\thinmuskip{;}\mskip6mu plus1mu\relax}

\newcommand{\dual}{\mathrm{D}}

\newcommand{\defeq}{\mathbin{:=}}

\newcommand{\bl}{\bigl(}
\newcommand{\br}{\bigr)}
\newcommand{\To}[1][]{\xrightarrow[]{\mspace{10mu}{#1}\mspace{10mu}}}

\newenvironment{myarray}[1]{\relax\setlength{\arraycolsep}{1pt}

\begin{array}{#1}}{\end{array}\relax}

\newcommand{\ba}{\begin{myarray}}
\newcommand{\ea}{\end{myarray}}

\newcommand{\be}{\begin{enumerate}}
\newcommand{\ee}{\end{enumerate}}
\newcommand{\bnum}{\be[{\rm(i)}]}

\nc{\bwr}{\mbox{\large{$\wr$}}}
\nc{\vphi}{\varphi}
\nc{\seteq}{\mathbin{:=}}
\nc{\noi}{\noindent}
\nc{\ro}{{\rm(}}
\nc{\rf}{{\rm)}\xspace}
\nc{\ms}{\mspace}
\nc{\sbcup}{\mathop{\scalebox{0.75}{$\displaystyle\bigcup$}}}
\nc{\ol}{\overline}
\nc{\scbul}{{\,\raise1pt\hbox{$\scriptscriptstyle\bullet$}\,}}
\nc{\set}[2]{\left\{#1\;\semicolon\; #2 \right\}}
\nc{\extp}{\mathop{\raisebox{.3ex}{\scalebox{0.8}{$\displaystyle\bigwedge$}}}\limits}

\newenvironment{myequation}
{\relax\setlength{\arraycolsep}{1pt}\begin{eqnarray}}
{\end{eqnarray}}
\newenvironment{myequationn}
{\relax\setlength{\arraycolsep}{1pt}\begin{eqnarray*}}
{\end{eqnarray*}}

\newenvironment{myalign}
{\relax\begin{align}}
{\end{align}}
\newenvironment{myalignn}
{\relax\begin{align*}}
{\relax\end{align*}}

\nc{\eq}{\begin{myequation}}
\nc{\eneq}{\end{myequation}}
\nc{\eqn}{\begin{myequationn}}
\nc{\eneqn}{\end{myequationn}}

\nc{\eqa}{\begin{myalign}}
\nc{\eneqa}{\end{myalign}}
\nc{\eqan}{\begin{myalignn}}
\nc{\eneqan}{\end{myalignn}}

\nc{\on}{\operatorname}
\nc{\Ind}{\on{Ind}}
\nc{\Proof}{\begin{proof}}
\nc{\QED}{\end{proof}}
\nc{\cor}{\field}
\nc{\tone}{\To[+1]}

\renewcommand{\tor}{\mathrm{tor}}

\newcommand{\bordered}[1]{{\mathsf{#1}}}

\newcommand{\bclose}[1]{{\accentset{\vee}{#1}}}
\newcommand{\unbordered}[1]{{\accentset{\circ}{#1}}}

\newcommand{\bR}{{\R_\infty}}

\newcommand{\bM}{\inbordered{M}}
\nc{\unb}{\unbordered}
\nc{\eps}{\varepsilon}
\nc{\epsi}{\epsilon}
\nc{\inb}{\inbordered}
\nc{\colim}{\varinjlim\limits}
\nc{\ssubset}{\subset\ms{-3mu}\subset}
\nc{\al}{\alpha}
\nc{\qtq}[1][and]{\quad\text{#1}\quad}
\nc{\qt}[1]{\quad\text{#1}}
\nc{\olG}[1][f]{{\overset{\ms{4mu}\rule[-.05ex]{1.6ex}{.115ex}}{\Gamma}}_{%
\ms{-3mu}#1}}

\newcommand{\cM}{\bclose{M}}

\nc{\cf}{\bclose{f}}
\nc{\cp}{\bclose{p}}

\newcommand{\inbordered}[1]{{#1_\infty}}

\newcommand{\bZ}{\inbordered{Z}}

\newcommand{\quot}{\derived Q}

\newcommand{\Efield}{\field^\enh}

\newcommand{\W}{{V^*}}
\newcommand{\RP}{\mathsf{P}}

\newcommand{\ex}{\mathsf{E}}

\newcommand{\Fou}{\mathsf{F}}
\newcommand{\Foua}{{\mathpalette\Fouatemp\relax}}
\newcommand{\Fouatemp}[2]{\reflectbox{$#1\Fou$}}
\newcommand{\Lap}{\mathsf{L}}
\newcommand{\lap}{{}^\Lap}

\newcommand{\Lapa}{{\mathpalette\Lapatemp\relax}}
\newcommand{\Lapatemp}[2]{\reflectbox{$#1\Lap$}}
\newcommand{\lapa}{{}^{\Lapa}}
\newcommand{\lapar}{{}^{\Lapa^r}}

\newcommand{\bX}{\bordered{X}}

\newcommand{\bY}{\bordered{Y}}

\nc{\tM}{\widetilde{M}}
\nc{\tX}{\widetilde{X}}
\nc{\ti}{{\tilde\imath}}
\nc{\tj}{{\tilde\jmath}}

\newcommand{\dt}[1]{\widetilde{#1}^{\mathsf d}}
\newcommand{\dtd}[1]{\widetilde{#1}^{\mathsf d,\times}}
\nc{\dtM}{\dt M}
\nc{\dtdM}{\dtd M}
\nc{\dti}{\tilde\imath^{\mathsf d}}

\newcommand{\rmpt}{\mathrm{pt}}
\newcommand{\pt}{\st\rmpt}
\newcommand{\rb}{\mathsf{rb}}
\newcommand{\pb}{\mathsf{pb}}
\newcommand{\nd}{\mathsf{nd}}
\newcommand{\sm}{\mathsf{sm}}
\newcommand{\sph}{\mathsf{sph}}
\nc{\st}[1]{{\{{#1}\}}}
\nc{\bP}{\mathbb{P}}
\nc{\into}{\hookrightarrow}
\nc{\fR}{{\R_\infty}}
\nc{\cS}{\bclose{S}}
\nc{\RB}[2][N]{#2_{#1}^{\rb}}
\nc{\prb}{p_{\rb}}
\nc{\PB}[2][N]{#2_{#1}^{\pb}}
\nc{\ppb}{p_{\pb}}
\nc{\ND}[2][N]{#2_{#1}^{\nd}}
\nc{\pnd}{p_{\nd}}
\nc{\snd}{s_{\nd}}
\nc{\dND}[2][N]{\bdot{#2}_{#1}^\nd}
\nc{\psm}{p_{\sm}}

\newcommand{\Top}{\mathcal T\!op}
\newcommand{\bTop}{b\text{-}\Top}

\newcommand{\bdd}[1]{\mathsf{#1}}
\newcommand{\bds}{\bdd S}
\newcommand{\bdt}{\bdd T}
\newcommand{\bdx}{\bdd X}
\newcommand{\bdy}{\bdd Y}
\newcommand{\bdz}{\bdd Z}

\newcommand{\Vi}{V_\infty}
\newcommand{\dVi}{{\bdot V}_\infty}
\newcommand{\Wi}{V^*_\infty}

\newcommand{\tu}{\widetilde u}
\newcommand{\tx}{\widetilde x}

\newcommand{\Enu}{\enh\nu}
\newcommand{\Emu}{\enh\mu}
\newcommand{\Esm}{\enh\sigma}

\newcommand*\bigcdot{\mathpalette\bigcdot@{.5}}
\newcommand*\bigcdot@[2]{\mathbin{\vcenter{\hbox{\scalebox{#2}{$\m@th#1\bullet$}}}}}
\newcommand{\bdot}[1]{{\accentset{\bigcdot}{#1}}\vphantom{#1}}

\newcommand{\Gm}{\R^\times}
\newcommand{\Gmp}{\R^\times_{>0}}
\newcommand{\bGmp}{\inb{(\Gmp)}}

\newcommand{\XY}{{\rm prod}}

\newcommand{\SV}{\mathbb{S}V}

\newcommand{\bbM}{\mathsf{M}}
\newcommand{\ccM}{C}
\newcommand{\obM}{\unbordered{\bbM}}
\newcommand{\cbM}{\bclose{\bbM}}

\newcommand{\bbN}{\mathsf{N}}
\newcommand{\obN}{\unbordered{\bbN}}
\newcommand{\cbN}{\bclose{\bbN}}

\newcommand{\bbS}{\mathsf{S}}
\newcommand{\obS}{\unbordered{\bbS}}
\newcommand{\cbS}{\bclose{\bbS}}

\newcommand{\oo}{\unbordered}
\newcommand{\of}{\oo f}
\nc{\oloG}[1][\of]{{\overset{\ms{4mu}\rule[-.05ex]{1.6ex}{.115ex}}{\Gamma}}_{%
\ms{-3mu}#1}}

\nc{\ake}[1][2ex]{\rule[-.5ex]{0ex}{#1}}
\nc{\akew}[1][2ex]{\rule[-1ex]{#1}{0ex}}
\nc{\aked}[2][2ex]{\rule[{#1}]{0ex}{#2}}

\nc{\la}{\lambda}

\makeatother

\begin{document}

\title{Enhanced specialization and microlocalization}

\author[A.~D'Agnolo]{Andrea D'Agnolo}
\address[Andrea D'Agnolo]{Dipartimento di Matematica\\
Universit{\`a} di Padova\\
via Trieste 63, 35121 Padova, Italy}
\thanks{The research of A.D'A.\
was partially supported by GNAMPA/INdAM.  He acknowledges the kind hospitality at RIMS of
Kyoto University during the preparation of this paper.}
\email{dagnolo@math.unipd.it}

\author[M.~Kashiwara]{Masaki Kashiwara}
\thanks{The research of M.K.\
was supported by Grant-in-Aid for Scientific Research (B)
15H03608, Japan Society for the Promotion of Science}
\address[Masaki Kashiwara]{Research Institute for Mathematical Sciences, Kyoto University,
Kyoto 606-8502, Japan \& Korea Institute for Advanced Study, Seoul 02455, Korea}
\email{masaki@kurims.kyoto-u.ac.jp}



\maketitle

\begin{abstract}
Enhanced ind-sheaves provide a suitable framework for the irregular Riemann-Hilbert correspondence.
In this paper, we show how Sato's specialization and microlocalization functors have a natural enhancement, and discuss some of their properties.
\end{abstract}

\tableofcontents


\section{Introduction}

\subsection{}\label{sse:introMN}
Let $M$ be a real analytic manifold, and $N\subset M$ a closed submanifold. 
The normal deformation (or deformation to the normal cone) of $M$ along $N$ is a real analytic manifold $\ND M$ endowed with a map $(p,s)\colon \ND M\to M\times\R$, such that $\opb s(\R_{\neq 0}) \isoto M\times\R_{\neq 0}$ and $\opb s(0)$ is identified with the normal bundle $T_NM$.

Sato's specialization functor $\nu_N$, defined through $p\colon\ND M\to M$, associates to a sheaf\footnote{We abusively call \emph{sheaf} an object of the bounded derived category $\BDC(\field_M)$ of sheaves of $\field$-vector spaces on $M$, for a fixed base field $\field$.} $F\in\BDC(\field_M)$ a conic sheaf on $T_NM$, describing the asymptotic behaviour of $F$ along $N$. 
Let $\bdot T_NM$ be the complement of the zero section, identified with $N$.
One has $\nu_N(F)|_N\simeq F|_N$, and $\nu_N(F)|_{\bdot T_NM}$ only depends on $F|_{M\setminus N}$.

Sato's microlocalization functor $\mu_N$ is obtained from $\nu_N$ by Fourier-Sato transform, and provides a tool for the microlocal analysis of $F$ on the conormal bundle $T^*_NM$.

\subsection{}\label{sse:startdef}
In this paper, we will define the enhanced version of 
the specialization and microlocalization functors.
With notations as in \S\ref{sse:introMN}, this proceeds as follows.

We start by showing that there exists a (unique) real analytic bordered space $\inb{(\ND M)}$ such that the map $p\colon\inb{(\ND M)}\to M$ is semiproper. 
(We call this a bordered compactification of $p$.)

Mimicking the classical construction, with $\ND M$ replaced by $\inb{(\ND M)}$, we get an enhancement of Sato's specialization.
This associates to an enhanced ind-sheaf $K\in\BEC M$ a conic enhanced ind-sheaf  $\Enu_N(K)$ on the bordered compactification $\inb{(T_NM)}$ of $T_NM\to N$.
Consider the bordered space $\inb{(M\setminus N)}\defeq(M\setminus N,M)$.
One has $\Enu_N(K)|_N\simeq K|_N$, and $\Enu_N(K)|_{\inb{(\bdot T_NM)}}$ depends on $K|_{\inb{(M\setminus N)}}$, not only on $K|_{M\setminus N}$.

Then, using the enhanced Fourier-Sato transform $\lap(\cdot)$, we get the enhanced microlocalization functor $\Emu_N\defeq\lap\Enu_N$, with values in conic enhanced ind-sheaves on $\inb{(T^*_NM)}$.

We establish some functorial properties of the functors $\Enu_N$ and $\Emu_N$. These are for the most part analogous to properties of the classical functors $\nu_N$ and $\mu_N$, but often require more geometrical proofs.

\subsection{}
In view of future applications to the Fourier-Laplace transform of holonomic $\D$-modules, we also discuss the following situation.

Let $\tau\colon V\to N$ be a vector bundle, and $\inb V$ its bordered compactification.
In this setting, we consider the natural enhancement of the smash functor from \cite[\S6.1]{DHMS18}, described as follows. Let $\SV\defeq\bl(\R\times V)\setminus(\st0\times N)\br/\Gmp$ be the fiberwise sphere compactification of $\tau$, 
and identify $V$ with the hemisphere $(\R_{>0}\times V)/\Gmp$. Consider the hypersurface $H\defeq\partial V\subset \SV$, and identify $\bdot V\seteq V\setminus N$
with the half of $ T_H(\SV)$ pointing to $V$. Then, the restriction of $\Enu_H$ to $\inb{(\bdot V)}$ can be thought of as a ``specialization at infinity'' on $V$. The enhanced smash functor $\Esm_V$ (see \S\ref{sse:smash}) 
provides an extension of $\Enu_H$ from $\inb{(\bdot V)}$ to $\inb V$.

If $K$ is an enhanced ind-sheaf on $\inb V$, with the natural identification $\inb{V^*}\simeq\inb{(T^*_NV)}$ one has (see Proposition~\ref{pro:musm})
\[
\Esm_\W(\lap K) \simeq \Emu_N(K).
\]

\subsection{}
The contents of this paper are as follows.

We introduce in Section~\ref{se:bord} the notion of bordered compactification. This is a relative analogue of the classical one-point compactification. Then, we show that the normal deformation $p\colon\ND M\to M$ has a bordered compactification $\inb{(\ND M)}$ in the category of subanalytic bordered spaces.

Using the bordered normal deformation, and after recalling some notations in Section~\ref{se:not}, the enhanced specialization is introduced and studied in Section~\ref{se:nu}. 
We also consider an analogous construction, attached to the real oriented blow-up of $M$ with center $N$. Moreover, we discuss the notion of conic enhanced ind-sheaf.

Section~\ref{se:mu} establishes some complementary results on the enhanced Fourier-Sato transform, and uses it to enhance the microlocalization functor. 
Finally, in Section~\ref{se:smash}, we link the microlocalization along the zero section of a vector bundle with the so-called smash functor.

\section{Bordered normal deformation}\label{se:bord}

Here, after recalling the notion of bordered space from \cite[\S3]{DK16}, we introduce the notion of bordered compactification. We then show that the deformation to the normal cone, for which we refer to \cite[\S4.1]{KS90}, admits a canonical bordered compactification. 

\medskip
In this paper, a good space is a topological space which is Hausdorff, locally compact, countable at infinity, and with finite soft dimension. 

\subsection{Bordered spaces}
Denote by $\Top$ the category of good spaces and continuous maps. 

Denote by $\bTop$ the category of bordered spaces, whose objects are pairs $\bbM=(M,\ccM)$ with $M$ an open subset of a good space $C$. Set $\obM\defeq M$ and $\cbM\defeq C$.
A morphism $f\colon\bbM\to\bbN$ in $\bTop$ is a morphism $\of\colon \obM\to \obN$ in $\Top$ such that the projection $\oloG\to \cbM$ is proper. Here, $\oloG$ denotes the closure in $\cbM\times\cbN$ of the graph $\Gamma_{\oo f}$ of $\of$.

The functor $\bbM\mapsto\obM$ is right adjoint to the embedding $\Top\to\bTop$, $M\mapsto(M,M)$.  
We will write for short $M=(M,M)$. 
Note that $\bbM\mapsto\cbM$ is not a functor.

We say that $f\colon\bbM\to\bbN$ is
{\em semiproper} if
$\oloG\to \cbN$ is proper.
We say that $\bbM$ is semiproper if so is the natural morphism $\bbM\to\pt$.
We say that $f\colon\bbM\to\bbN$ is {\em proper} if
it is semiproper and $\oo f\colon \obM\to \obN$ is proper.

For any bordered space $\bbM$ there are canonical morphisms
\[
\xymatrix{
\obM \ar[r]^{i_\bbM} &\bbM  \ar[r]^{j_\bbM} & \cbM.
}
\]
Note that $j_\bbM$ is semiproper.

By definition, a subset $Z$ of $\bbM$ is a subset of $\obM$.
We say that $Z$ is open (resp.\ locally closed) if it is so in $\obM$.
For a locally closed subset $Z$ of $\bbM$, we set $\bZ=(Z,\overline Z)$ where $\overline Z$ is the closure of $Z$ in $\cbM$.  Note that, for an open subset $U\subset \bbM$, we have $\inb U \simeq (U,\cbM)$. 

We say that $Z$ is relatively compact in $\bbM$ if it is contained in a compact subset of $\cbM$.
Then, for a morphism of bordered spaces   $f\cl \bbM\to \bbN$,
the image $\oo f(Z)$ is relatively compact in $\bbN$.
In particular, the condition of $Z$ being relatively compact in $\bbM$ does not depend on the choice of $\cbM$.

\subsection{Bordered compactification}\label{sse:bordcomp}

Let $\bbS$ be a bordered space.
Denote by $\bTop_{\,\bbS}$ the category of bordered spaces over $\bbS$, and by $\mathcal T\!op_{\,\obS}$ the category of good spaces over $\obS$.

\Lemma\label{lem:semiproper}
Let $\bbM$ and
$\bbN$ be bordered spaces over $\bbS$.
If $\bbN$ is semiproper over $\bbS$,
then the natural morphism
\[
\Hom[\bTop_{\,\bbS}](\bbM,\bbN) \to \Hom[\mathcal T\!op_{\,\obS}](\obM,\obN)
\]
is an isomorphism.
\enlemma

\begin{proof}
Denote by $p\colon\bbM\to\bbS$ and $q\colon\bbN\to\bbS$ the given morphisms.
In order to prove the statement, it is enough to show that any  continuous map $\oo f\colon\obM\to\obN$ which enters the commutative diagram
\[
\xymatrix@R=1ex{
\obM \ar[rr]^{\oo f} \ar[dr]_{\oo p} && \obN \ar[dl]^{\oo q} \\ & \obS
}
\]
induces a morphism $f\colon\bbM\to\bbN$.
 That is, we have to prove that the map $\oloG\to\cbM$ is proper.

It is not restrictive to assume that $\oo p$ extends to a map $\cbM\to\cbS$.
Since $\Gamma_{\oo f}\subset \obM\times \obN$ is included in $\obM\times_\obS \obN$, its closure $\oloG$ in $\cbM\times\cbN$ is included in $\cbM\times_\cbS\olG[\oo q]$. Since $q$ is semiproper, the map $\olG[\oo q]\to\cbS$ is proper. Hence so is the map $\oloG\to\cbM$.
\end{proof}

\Prop\label{pro:bordcomp}
Let $M$ be a good space, and $p\cl M\to \bbS$ a morphism of bordered spaces.
Then there exists a bordered space
$\bM$, with $(M_\infty)^\circ=M$, such that $\oo p$ induces a semiproper morphism $\inb p\cl \bM\to\bbS$.
Such an $\bM$ is unique up to a unique isomorphism.
\enprop

\begin{definition}
With notations as above, $\bM$ is called the {\em bordered compactification} of $M$ over $\bbS$.
\end{definition}

\Proof[Proof of Proposition~\ref{pro:bordcomp}]
Set $\cM\defeq M\sqcup \cbS$, and endow it with the following topology.
Let $i\cl \cbS\into \cM$ and $j\cl M\into \cM$ be the inclusions.
Consider the map $\oo p\colon M\to\obS$.
For $s\in \cbS$, a neighborhood of $i(s)$ is a subset of $\cM$ containing
$i(V)\cup j\bl \oo p{}^{-1}(V\cap \obS)\setminus K\br$, where $V\subset \cbS$ is a neighborhood 
of $s$, and $K\subset M$ is a compact subset.
For $x\in M$, a neighborhood of $j(x)$ is a subset of $\cM$ containing
$j(U)$, where $U\subset M$ is a neighborhood of $x$.
It is easy to check that $\cM$ is a good topological space containing $M$ as an open subset.

Define $\cp\colon\cM\to\cbS$ by $\cp(j(x))=\oo p(x)$ for $x\in M$, and  $\cp(i(s))=s$ for $s\in \cbS$. Then, $\cp$ is proper.
It follows that, setting $\bM\defeq(M,\cM)$, the morphism $p$ extends to a semiproper morphism $\bM\to \cbS$.
Such a morphism factors as $\bM\To[p_\infty] \bbS\To[j_{\bbS}] \cbS$, and hence also $p_\infty$ is semiproper.

This proves the existence. Uniqueness follows from Lemma~\ref{lem:semiproper}, by considering the commutative diagram
\[
\xymatrix@R=1ex{
M \ar[rr]^{i_M} \ar[dr]_{p} && \bM \ar[dl]^{p_\infty} \\ & \bbS.
}
\]
\QED

\subsection{Blow-ups and normal deformation}\label{sse:blownorm}
Let $M$ be a real analytic manifold and $N\subset M$ a closed submanifold.
Denote by $\tau\colon T_NM\to N$ the normal bundle, and by $\bdot T_NM\subset T_NM$ the complement of the zero-section.
Recall that the multiplicative groups $\Gm\defeq\R_{\neq 0}$ and $\Gmp\defeq\R_{> 0}$  act freely on $\bdot T_NM$. Denote by $S_NM\defeq\bdot T_NM/\Gmp$ the sphere normal bundle, and by $P_NM\defeq\bdot T_NM/\Gm$ the projective normal bundle.

\begin{notation}\label{not:blowND}
\begin{itemize}
\item[(i)]
Denote by  $\prb\cl\RB{M}\to M$ the real oriented blow-up of $M$ with center $N$.
Recall that $\RB{M}$ is a subanalytic space, that $\prb$
induces an isomorphism
$\prb^{-1}(M\setminus N)\isoto M\setminus N$, and that
$\prb^{-1}(N)=S_NM$. In fact, $\RB M$ is a real analytic manifold with boundary $S_NM$.
This is pictured in the commutative diagram
\begin{equation}
\label{eq:RBMdiagram}
\xymatrix@R=3ex{
S_NM \ar@{^(->}[r] \ar[d] & \RB M \ar[d]_{\prb} & M\setminus N \ar@{_(->}[l] \ar@{_(->}[dl] \\
N \ar@{^(->}[r] \ar@{}[ur]|-\square & M\,.
}
\end{equation}
\item[(ii)]
Denote by $\ppb\cl\PB{M}\to M$ the real projective blow-up of $M$ with center $N$. Recall that $\PB{M}$ is a real analytic manifold, that $\ppb$
induces an isomorphism
$\ppb^{-1}(M\setminus N)\isoto M\setminus N$, and that
$\ppb^{-1}(N)=P_NM$. 
This is pictured in the commutative diagram
\begin{equation}
\label{eq:PBMdiagram}
\xymatrix@R=3ex{
P_NM \ar@{^(->}[r] \ar[d] & \PB M \ar[d]_{\ppb} & M\setminus N \ar@{_(->}[l] \ar@{_(->}[dl] \\
N \ar@{^(->}[r] \ar@{}[ur]|-\square & M\,.
}
\end{equation}
\end{itemize}
We have a natural commutative diagram
\[
\xymatrix@R=2ex@C=1em{
\PB M \ar[dr]_\ppb && \RB M \ar[dl]^\prb \ar[ll] \\
&M.
}
\]
\begin{itemize}
\item[(iii)]
Denote by $(\pnd,\snd)\cl \ND{M}\to M \times \R$ the normal deformation (or deformation  to the normal cone) of $M$ along $N$ (see \cite[\S4.1]{KS90}).
Recall that  $\ND{M}$ is a real analytic manifold, and that $(\pnd,\snd)$
induces isomorphisms $\pnd^{-1}(M\setminus N) \isoto (M\setminus N)\times\R_{\neq 0}$ and $\snd^{-1}(\R_{\neq 0}) \isoto M\times\R_{\neq 0}$. 
One also has $\snd^{-1}(\st{0})=T_NM$.
This is pictured in the commutative diagram
\begin{equation}
\label{eq:NDMdiagram}
\xymatrix@R=4ex@C=8ex{
T_NM \ar@{^(->}[r]^-{i_\nd} \ar[d] & \ND M \ar[d]_-{(\pnd,\snd)} & M\times\R_{\neq 0} \ar@{_(->}[l] \ar@{_(->}[dl]\\
M\times\st0 \ar@{^(->}[r] \ar@{}[ur]|-\square & M\times\R\,.
}
\end{equation}
There is a natural action of $\Gm$ on $\ND M$, extending that on $T_NM$.
The map $\snd\colon\ND M\to\R$ is smooth and equivariant with respect to the action of $\Gm$ on $\R$ given by $c\cdot s = c^{-1}s$. 
For $\Omega\seteq\opb\snd(\R_{>0})\isoto M\times\R_{>0}$, consider the commutative diagram
\begin{equation}
\label{eq:normdefMN}
\xymatrix@R=3ex@C=8ex{
T_NM \ar@{^(->}[r]^-{i_\nd} \ar[d]_\tau & \ND M \ar[d]_{\pnd} & \Omega \ar@{_(->}[l]_-{j_\nd} \ar[dl]^-{p_\Omega} \\
N \ar@{^(->}[r]_-{i_N} \ar@{}[ur]|-\square & M\,,
}
\end{equation}
where we set $p_\Omega\defeq\pnd|_\Omega$.
Note that $\overline\Omega=\opb\snd(\R_{\geq0})=\Omega\dunion T_NM$.
For $S\subset M$, the \emph{normal cone} to $S$ along $N$ is defined by
\begin{equation}\label{eq:CN}
C_N(S) \defeq T_NM\cap\overline{p_\Omega^{-1}(S)}.
\end{equation}
\item[(iv)] 
Denote by $\widetilde\Omega$ the complement of $\opb\pnd(N)\setminus\bdot T_NM$ in $\overline\Omega$, i.e.\ $\widetilde\Omega=\bl(M\setminus N)\times\R_{>0}\br\dunion\bdot T_NM$.
Thus, $\widetilde\Omega$ is an open subset of $\overline\Omega$ which is invariant by the action of $\Gmp$, and enters the commutative diagram
\begin{equation}
\label{eq:dNDMdiagram}
\xymatrix@R=3ex{
T_NM \ar@{^(->}[r] & \overline\Omega \ar[dr]^-{\pnd|_{\overline\Omega}} \\
\bdot T_NM \ar@{^(->}[r] \ar@{}[ur]|-\square \ar@{^(->}[u] \ar[d] & \widetilde\Omega \ar@{^(->}[u] \ar[d]^-\gamma  & M \,. \\
S_NM \ar@{^(->}[r] \ar@{}[ur]|-\square & \RB M  \ar[ur]_{\prb}
}
\end{equation}
Note that $\gamma\colon\widetilde\Omega\to\RB M$ is a principal $\Gmp$-bundle.
\end{itemize}
\end{notation}

\begin{remark}\label{re:blowcoord}
Let us illustrate the above constructions in local coordinates.
Consider a chart $M\supset U\to[\varphi]\R^m_x\times\R^n_y$ such that $N\cap U=\varphi^{-1}(\st{x=0})$. 

\smallskip\noindent (i)
Let $\Gmp$ act on $\R^m_v\times\R^n_y\times\R_s$ by $c\cdot(v,y,s)=(c v,y,c^{-1}s)$.
Then $\prb^{-1}(U)\subset\RB M$ has $\Gmp$-homogeneous coordinates $[v,y,s]$ with $v\neq0$, $s\geq 0$, and $(sv,y)\in\varphi(U)$. One has $\prb([v,y,s])=(sv,y)$.

\smallskip\noindent (ii)
Similarly, replacing the action of $\Gmp$ by that of $\Gm$, the open subset $\ppb^{-1}(U)\subset\PB M$ has $\Gm$-homogeneous coordinates $[v,y,s]$ with $v\neq0$ and $(sv,y)\in\varphi(U)$. One has $\ppb([v,y,s])=(sv,y)$.

\smallskip\noindent (iii)
The open subset $\pnd^{-1}(U)\subset\ND M$ has coordinates $(v,y,s)\in\R^{m+n+1}$, with $(sv,y)\in\varphi(U)$. One has $\pnd(v,y,s)= (sv,y)$ and $\snd(v,y,s)=s$.
The action of $\Gm$ on $\ND M$ is given by $c\cdot(v,y,s)=(c v,y,c^{-1}s)$. One has $\Omega\cap\pnd^{-1}(U)=\st{s>0}$ and $\overline\Omega\cap\pnd^{-1}(U)=\st{s\geq0}$.

\smallskip\noindent (iv)
One has $\widetilde\Omega\cap\pnd^{-1}(U)=\{(v,y,s)\semicolon s\geq0,\ v\neq 0\}$ and
$\gamma(v,y,s)=[v,y,s]\in\RB M$.
\end{remark}

\subsection{Bordered normal deformation}\label{se:bdnd}
Let $M$ be a real analytic manifold and $N\subset M$ a closed submanifold.
Set $X=M\times \RP$ and $Y=N\times\st0\subset X$, where
$\RP\seteq\R\cup\{\infty\}$ is the real projective line.
There is a natural commutative diagram
\begin{equation}\label{eq:NDPB}
\xymatrix@R=3ex{
\ND M \ar@{^(->}[r] \ar[d]_{(\pnd,\snd)} & \PB[Y]{X} \ar[d]^{\ppb} \\
M \times\R \ar@{^(->}[r] \ar@{}[ur]|-\square & M \times\RP\rlap{${}=X,$}
}
\end{equation}
where the bottom arrow is induced by the inclusion of the affine chart $\R\subset\RP$, and the top arrow is the embedding described as follows.
Recall that $\ND M=\opb\snd(\R_{\neq 0})\dunion T_NM$.
The natural identifications $\snd\colon\snd^{-1}(\R_{\neq0})\isoto M\times\R_{\neq0}$ and $\ppb\colon\ppb^{-1}(X\setminus Y) \isoto (M\times\RP)\setminus(N\times\st0)$, provide an open embedding $\snd^{-1}(\R_{\neq0}) \subset \PB[Y]X$.
This extends to $\ND[N]M$ by sending $v\in T_NM$ to $[v,1]\in P_YX = \ppb^{-1}(Y)$.
Note that one has
\begin{align*}
\PB[Y]X\setminus\ND{M}
&= \ppb^{-1}\bl M\times\st\infty\br\sqcup\overline{\ppb^{-1}\bl (M\setminus N)\times\st 0\br} \\
&= \ppb^{-1}\bl M\times\st\infty\br\sqcup\ppb^{-1}\bl (M\setminus N)\times\st 0\br \\
&\phantom{= \ppb^{-1}\bl M\times\st\infty\br\,}\sqcup P_{N\times\st0}(M\times\st0).
\end{align*}

\begin{remark}
Let us describe the above constructions in the situation of Remark~\ref{re:blowcoord}.
Consider the action of $\Gm$ on $\R^m_v\times\R^n_y\times\R_r\times\R_s$ given by $c\cdot(v,y,r,s)=(c v,y,c r,c^{-1}s)$.
Let $\R=\RP\setminus\{\infty\}$ be the affine chart.
Then $\ppb^{-1}(U\times\R)\subset\PB[Y]X$ has $\Gm$-homogeneous coordinates $[v,y,r,s]$ with $(v,s)\neq(0,0)$ and $(sv,y)\in\varphi(U)$. 
One has $\ppb([v,y,r,s])=(sv,y,sr)$.
The embedding $\ND M\into \PB[Y]{X}$ is given by $(v,y,s)\mapsto[v,y,1,s]$.
\end{remark}

Recall from \cite[\S5.4]{DK16} that a real analytic bordered space is a bordered space $\bbM$ such that $\cbM$ is a real analytic manifold, and $\obM\subset\cbM$ is a subanalytic open subset. A morphism $f\colon\bbM\to\bbN$ of real analytic bordered spaces is a morphism of bordered spaces such that $\oo f$ is a real analytic map, and $\overline{\Gamma_{\oo f}}$ is a subanalytic subset of $\cbM\times\cbN$.

\begin{lemmadefinition}
The bordered compactification of $\ND{M}$ over $M$ has a realization in the category of real analytic bordered spaces by $\inb{(\ND M)} \defeq (\ND M,\PB[Y]X)$, using the open embedding \eqref{eq:NDPB}.
Note that the projection $\PB[Y]X\to M$ is proper.
\end{lemmadefinition}

Note that the closure $\overline{T_NM}$ of $T_NM$ in $P_YX$ is the projective compactification of $T_NM$ along the fibers of $\tau\colon T_NM\to N$.
Considering the bordered spaces $\inb{(T_NM)}\seteq(T_NM,\overline{T_NM})$ and $\inb\Omega\seteq(\Omega,\PB[Y]X)$, one has
the commutative diagram with Cartesian squares of bordered spaces semiproper over $M$
\begin{equation}
\label{eq:Mndinfty}
\xymatrix@R=4ex{
\inb{(T_NM)} \ar@{^(->}[r] \ar[d] & \inb{(\ND M)} \ar[d]^{(\pnd,\snd)} & \inb\Omega \ar@{_(->}[l] \ar[d]^\bwr \\
N\times\st0 \ar@{^(->}[r] \ar@{}[ur]|-\square & M\times\bR \ar@{}[ur]|-\square & \ar@{_(->}[l] M\times\inb{(\R_{>0})}\,.
}
\end{equation}
Note that the morphisms in the top row are $\bGmp$-equivariant.
Here, $\bGmp\seteq(\R_{>0},\ol\R)$ is a group object in $\bTop$.

\section{Review on enhanced ind-sheaves}\label{se:not}

We recall here some notions and results, mainly to fix notations, referring to the literature for details. In particular, we refer to \cite{KS90} for sheaves, to \cite{Tam18} (see also \cite{GS14,DK19}) for enhanced sheaves, to \cite{KS01} for ind-sheaves, and to \cite{DK16} (see also \cite{KS16L,KS16D,Kas16,DK19}) for bordered spaces and enhanced ind-sheaves.

\medskip

In this paper, $\field$ denotes a base field.

\subsection{Sheaves}

Let $M$ be a good space.

Denote by $\BDC(\field_M)$ the bounded derived category of sheaves of $\field$-vector spaces on $M$, and by $\tens$, $\opb f$, $\reim f$ and $\rhom$, $\roim f$, $\epb f$ the six operations. Here $f\colon M\to N$ is a morphism of good spaces.

For $S\subset M$ locally closed, we denote by $\field_S$ the extension by zero to $M$ of the constant sheaf on $S$ with stalk $\field$.

\subsection{Ind-sheaves}

Let $\bbM$ be a bordered space.

We denote by $\BDC(\ifield_\bbM)$ the bounded derived category of ind-sheaves of $\field$-vector spaces on $\bbM$, and by $\tens$, $\opb f$, $\reeim f$ and $\rihom$, $\roim f$, $\epb f$ the six operations. Here $f\colon \bbM\to \bbN$ is a morphism of bordered spaces.

We denote by $\iota_\bbM\colon\BDC(\field_\obM)\to\BDC(\ifield_\bbM)$ the natural embedding,
by $\alpha_\bbM$ the left adjoint of $\iota_\bbM$.
One sets $\rhom\defeq\alpha_\bbM\rihom$.

For $F\in \BDC(\field_\obM)$,
we often write simply $F$ instead of $\iota_\bbM F$
in order to make notations less heavy.

\subsection{Enhanced ind-sheaves}
Denote by $t\in\R$ the coordinate on the affine line, consider the two-point compactification $\overline\R\seteq\R\cup\st{-\infty,+\infty}$, and set $\bR\defeq(\R,\overline\R)$.
For $\bbM$ a bordered space, consider the projection 
\[
\pi_\bbM\colon \bbM\times\bR\to \bbM.
\]

Denote by $\BEC \bbM\defeq\BDC(\ifield_{\bbM\times\bR})/\opb\pi_\bbM\BDC(\ifield_\bbM)$ the bounded derived category of enhanced ind-sheaves of $\field$-vector spaces on $\bbM$. Denote by $\quot_\bbM\colon\BDC(\ifield_{\bbM\times\bR})\to\BEC\bbM$ the quotient functor.

For $f\colon \bbM\to \bbN$ a morphism of bordered spaces, set
\[
f_\R \defeq f\times\id_{\bR}\colon \bbM\times\bR\to \bbN\times\bR.
\]
Denote by $\ctens$, $\Eopb f$, $\Eeeim f$ and $\cihom$, $\Eoim f$, $\Eepb f$ the six operations for enhanced ind-sheaves.  
Recall that $\ctens$ is the additive convolution in the $t$ variable, and that the external operations are induced via $\quot$ by the corresponding operations for ind-sheaves, with respect to the morphism $f_\R$.
Denote by $\Edual_M$ the Verdier dual.

There is a natural decomposition $\BEC \bbM \simeq \BECp \bbM \dsum \BECm \bbM$,
there are embeddings
\[
\epsilon^\pm_\bbM\colon\BDC(\ifield_\bbM) \rightarrowtail \BECpm\bbM, \quad
F\mapsto \quot_\bbM\bl\field_{\st{\pm t\geq 0}}\tens\opb{\pi_\bbM}F\br,
\]
and one sets $\epsilon_\bbM(F)\defeq\epsilon_\bbM^+(F)\dsum\epsilon_\bbM^-(F)\in\BEC\bbM$.
Note that $\epsilon_\bbM(F)\simeq\quot_\bbM\bl\field_{\st{t= 0}}\tens\opb{\pi_\bbM}F\br$.

\subsection{Stable objects}\label{sse:stable}
Let $\bbM$ be a bordered space.
Set
\begin{align*}
\field_{\{t\gg0\}} &\defeq \indlim[a\rightarrow+\infty]\field_{\{t\geq a\}} \in\BDC(\ifield_{\bbM\times\bR}), \\
\Efield_\bbM &\defeq \quot_\bbM\field_{\{t\gg0\}} \in \BECp\bbM.
\end{align*}
An object $K\in\BEC\bbM$ is called \emph{stable} if $K\isoto \Efield_\bbM\ctens K$.

There is an embedding
\[
e_\bbM\colon \BDC(\ifield_\bbM) \rightarrowtail \BECp\bbM, \quad F\mapsto \Efield_\bbM \ctens \epsilon_\bbM(F),
\]
with values in stable objects.

\section{Specialization}\label{se:nu}

We discuss here the natural enhancement of the notions of conic object and Sato's specialization. For the corresponding classical notions we refer to \cite[\S3.7]{KS90} and \cite[\S4.2]{KS90}, respectively. We also link the specialization functor with the real oriented blow-up.

\subsection{Conic objects}
Recall that the bordered space $\bGmp \seteq (\R_{>0},\overline\R)$ is  semiproper and has a structure of  bordered group (i.e., is a group object in the category of bordered spaces).
Let $\bbM$ be a bordered space endowed with an action of $\bGmp$, and consider the maps
\[
p,\mu\colon \bbM\times\bGmp \to \bbM,
\]
where $p$ is the projection and $\mu$ is the action.
Similarly to \cite{KS90}, one says that an object $K\in\BEC \bbM$ is $\bGmp$-conic if there is an isomorphism
\[
\Eopb p K \simeq \Eopb\mu K.
\]
(Recall that if $\Eopb p K$ and $\Eopb\mu K$ are isomorphic, then there exists a unique isomorphism which restricts to the identity on $\bbM\times\{1\}$.)
Denote by $\BECcon \bbM$ the full triangulated subcategory of conic objects.

We say that a morphism  $\gamma\colon\bbM \to S$ is a principal $\bGmp$-bundle if  it is semiproper and if $\bbM$ is endowed with an action of $\bGmp$ such that
the underlying map $\oo\gamma\colon \obM \to S$ is a principal $\Gmp$-bundle.

\begin{lemma}\label{lem:Mprinc}
Let $\gamma\colon\bbM \to S$ be a principal $\bGmp$-bundle. Then, 
for $K\in\BECcon\bbM$ 
\begin{itemize}
\item[(i)]
one has
\[
K \simeq \Eopb\gamma \Eoim\gamma K\simeq \Eepb\gamma \Eeeim\gamma K.
\]
In particular, $K\simeq\Eopb\gamma H$ for some $H\in\BEC S$.
\item[(ii)]
One has $\Eeeim\gamma K\simeq\Eoim\gamma K[-1]$.
\end{itemize}
\end{lemma}

\begin{proof}
(i)
Since the proofs are similar, let us only discuss the first isomorphism.
Consider the cartesian diagram
\[
\xymatrix@R=3ex{
\bbM\times\bGmp \ar[r]^-p \ar[d]_\mu & \bbM \ar[d]^\gamma \\
\bbM \ar[r]_\gamma \ar@{}[ur]|-\square & S.
}
\]
Recalling that $K$ is $\bGmp$-conic, one has
\[
\Eopb\gamma \Eoim\gamma K
\simeq \Eepb\gamma \Eoim\gamma K[-1]
\simeq \Eoim p \Eepb\mu K[-1]
\simeq \Eoim p \Eepb p K[-1].
\]
Then, Sublemma~\ref{sub:pp} implies 
\begin{align*}
\Eopb\gamma \Eoim\gamma K
&\simeq \cihom\bl\epsilon_\bbM(\reim p\field_{\obM\times\R}),
K\br[-1] \\
&\simeq \cihom\bl\epsilon_\bbM(\field_\bbM[-1]),K\br[-1] 
\simeq K.
\end{align*}
\smallskip\noindent(ii)
One has 
\[
\Eeeim\gamma K
\simeq\Eeeim\gamma\Eopb\gamma\Eoim\gamma K
\simeq\Eoim\gamma K\ctens\epsilon_S\bl\reim \gamma\field_\obM\br
\simeq\Eoim\gamma K[-1],
\]
where the first isomorphism follows from (i), and the second isomorphism follows from Sublemma~\ref{sub:pp}.
\end{proof}

\begin{sublemma}\label{sub:pp}
Let $f\colon\bbM\to\bbN$ be a semiproper morphism of bordered spaces.
Then, for any $K\in\BEC \bbM$ one has
\begin{align*}
\Eeeim f\Eopb f K &\simeq K\ctens\epsilon_\bbN(\reim f\field_\obM), \\
\Eoim f\Eepb f K &\simeq \cihom\bl\epsi_\bbN(\reim f\field_\obM),K\br.
\end{align*}
\end{sublemma}

\begin{proof}
The first isomorphism follows from
\begin{align*}
\Eeeim f\Eopb f K 
&\simeq \Eeeim f \bl \Eopb f K\ctens \epsilon_\bbM(\field_\obM) \br \\
&\simeq  K\ctens \Eeeim f\bl\epsilon_\bbM(\field_\obM)\br  \\
&\simeq  K\ctens \epsilon_\bbN(\reim f\field_\obM),
\end{align*}
where the last isomorphism is due to the fact that $f$ is semiproper.

Similarly, the second isomorphism follows from
\begin{align*}
\Eoim f\Eepb f K 
&\simeq \Eoim f \cihom \bl \epsilon_\bbM(\field_\obM) , \Eepb f K \br \\
&\simeq \cihom \bl \Eeeim f\bl\epsilon_\bbM(\field_\obM)\br , K \br \\
&\simeq \cihom \bl \epsilon_\bbN(\reim f\field_\obM) , K \br.
\end{align*}
\end{proof}

\subsection{Conic objects on vector bundles}\label{sse:convect}
Let $\tau\colon V\to N$ be a real vector bundle over a good space $N$, and
let $\bdot V = V\setminus N$ be the complement of the zero-section.
Let $S_NV\to N$ be the associated sphere bundle defined by $S_NV\defeq\bdot V/\Gmp$. 
Consider the vector bundle $W\seteq\R\times V\to N$, and let $\bdot W \seteq W\setminus(\st0\times N)$ be the complement of the zero section.
The fiberwise sphere compactification $\SV\to N$ of $V\to N$ is the quotient $\SV\seteq\bdot W/\Gmp$, where the action is given by $c\cdot(u,x)=(c u,c x)$. The bordered compactification of $V\to N$ is given by $\inb V = (V,\,\SV)$. It is endowed with a natural $\bGmp$-action.
Consider the morphisms
\[
\xymatrix@C=8ex{
N \ar@{^(->}@<.5ex>[r]^--o & \inb V  \ar@<.5ex>[l]^(.4)\tau  & \inb{\bdot V} \ar[r]^--\gamma \ar@{_(->}[l]_j \ar@/^3ex/[ll]^-{\bdot\tau} & S_NV,
}
\]
where $o$ is the embedding of the zero section, $j$ is the open embedding, and $\gamma$ the quotient by the action of $\Gmp$.

\begin{notation}
For $K\in\BECcon{\inb V}$, set
\[
K^\sph \defeq \Eoim\gamma\Eopb j K \in \BEC{S_NV}.
\]
\end{notation}

\begin{lemma}\label{lem:Vconic}
For $K\in\BECcon{\inb V}$, one has the isomorphisms
\begin{itemize}
\item[(i)]
$\Eopb j K \simeq \Eopb \gamma K^\sph$,
\item[(ii)]
$\Eoim \tau K \simeq \Eopb o K$,
\item[(iii)]
$\Eeeim \tau K \simeq \Eepb o K$,
\end{itemize}
and a distinguished triangle
\begin{itemize}
\item[(iv)]
$\Eeeim {\bdot\tau} \Eopb\gamma K^\sph \to \Eepb o K \to \Eopb o K \to[+1]$.
\end{itemize}
\end{lemma}

\begin{proof}
(i) follows from Lemma~\ref{lem:Mprinc}.
 
\smallskip\noindent(ii) We will adapt some arguments in the proof of \cite[Lemma~2.1.12]{KSIW}.
One has 
\begin{align*}
\Eoim \tau \Eoim o \Eopb o K &\simeq \Eopb o K, \\
\Eoim \tau \Eeeim j\Eopb j K &\simeq \Eoim \tau \Eeeim j\Eopb\gamma\Eoim\gamma\Eopb j K,
\end{align*}
where the last isomorphism follows from Lemma~\ref{lem:Mprinc}.
Applying $\Eoim\tau$ to the distinguished triangle
\[
\Eeeim j\Eopb j K \to K \to \Eoim o \Eopb o K \to[+1],
\]
we are thus left to prove
\begin{equation}
\label{eq:tjgH}
\Eoim\tau \Eeeim j \Eopb\gamma H \simeq 0,
\end{equation}
for $H = \Eoim\gamma\Eopb j K$. Let us prove it for an arbitrary
$H \in \BEC{S_NV}$. 

Denoting by $\inb{(\RB[N]V)} \defeq \bl \RB[N] V, \RB\SV \br$ the bordered compactification of the real oriented blow-up $\prb\colon\RB[N] V\to \Vi$, consider the commutative diagram
\[
\xymatrix@R=3ex@C=8ex{
\inb{\bdot V} \ar@<.5ex>@/^1pc/[rr]^-\gamma \ar@{^(->}[dr]_-j \ar@{^(->}[r]_-{\tj}
 & \inb{(\RB[N] V)} \ar[r]_-{\tilde\gamma} \ar[d]^{\prb} &
S_NV \ar[d]^q \\
& \inb{V} \ar[r]^-\tau & N \,.
}
\]
Note that $\tilde\gamma\colon \RB[N] V\to S_NV$, $[x,r]\to {[x]}$, is an $\R_{\geq 0}$-fiber bundle. Hence one has (where we neglect for short the indices on $\epsilon$ and $\iota$)
\begin{align}
\label{eq:gammaepb}
\Eepb{\tilde\gamma}H 
&\simeq \epsilon(\epb{\tilde\gamma}\field_{S_NV})\ctens\Eopb{\tilde\gamma}H \\
\notag
&\simeq \epsilon(\field_{\tj(\bdot V)})\ctens\Eopb{\tilde\gamma}H[1], \\
\label{eq:gammarb}
\reim{\tilde\gamma} \field_{\RB[N] V} &\simeq 0.
\end{align}
Back to \eqref{eq:tjgH}, one has
\begin{align*}
\Eoim\tau \Eeeim j \Eopb\gamma H 
&\simeq \Eoim\tau \Eeeim{{\prb}} \Eeeim \tj \Eopb\gamma H \\
&\underset{(1)}\simeq \Eoim\tau \Eoim{{\prb}} \Eeeim \tj \Eopb\gamma H \\
&\simeq \Eoim q \Eoim{\tilde\gamma} \Eeeim \tj \Eopb\tj \Eopb{\tilde\gamma} H \\
&\simeq \Eoim q \Eoim{\tilde\gamma}\bl \epsilon(\field_{\tj(\bdot V)})\ctens\Eopb{\tilde\gamma}H \br\\
&\underset{(2)}\simeq \Eoim q \Eoim{\tilde\gamma} \Eepb{\tilde\gamma}H[-1] \\
&\underset{(3)}\simeq \Eoim q \cihom(\epsilon(\reim{\tilde\gamma}\field_{\RB[N] V}),H)[-1] \underset{(4)}\simeq 0,
\end{align*}
where $(1)$ is due to the fact that $p_\rb$ is proper, $(2)$ follows from \eqref{eq:gammaepb},  $(3)$ follows from Sublemma~\ref{sub:pp}
since $\tilde\gamma$ is semiproper, and $(4)$ follows from \eqref{eq:gammarb}.

\smallskip\noindent(iii) has a proof similar to (ii).

\smallskip\noindent(iv) 
Let us show that the distinguished triangle
\[
\Eeeim\tau\Eeeim j\Eopb j K \to
\Eeeim\tau K \to
\Eeeim\tau\Eeeim o\Eopb o K \to[+1]
\]
is isomorphic to the distinguished triangle in the statement.

\smallskip\noindent(iv-a) 
One has
$\Eeeim\tau\Eeeim j\Eopb j K
\simeq\Eeeim{\bdot\tau}\Eopb j K 
\simeq\Eeeim{\bdot\tau}\Eopb \gamma K^\sph$,
where the last isomorphism follows from (i).

\smallskip\noindent(iv-b) 
By (iii), one has
$\Eeeim\tau K \simeq \Eepb o K $.

\smallskip\noindent(iv-c) 
One has $\Eeeim\tau\Eeeim o\Eopb o K \simeq \Eopb o K$, since $\tau\circ o\simeq\id_N$.
\end{proof}

\begin{lemma}\label{lem:sphdual}
For $K\in\BECcon{\inb V}$, one has
\[
\Edual(K^\sph) \simeq \bl\Edual K\br^\sph[-1].
\]
\end{lemma}

\begin{proof}
One has 
\begin{align*}
\Edual(\Eoim\gamma\Eopb j K) 
&\underset{(*)}\simeq \Edual(\Eeeim\gamma\Eopb j K[1]) \\
&\simeq \Eoim\gamma\Eepb j \Edual(K)[-1] \\
&\simeq \Eoim\gamma\Eopb j \Edual(K)[-1],
\end{align*}
where $(*)$ follows from Lemma~\ref{lem:Mprinc} (ii).
\end{proof}

\subsection{Enhanced specialization}\label{sse:spec}

Let $M$ be a real analytic manifold, and $N\subset M$ a closed submanifold.
We will introduce here an enhancement of Sato's specialization functor. We refer to \cite[Chapter 4]{KS90} for the classical construction.

\medskip
Note that the action of $\Gmp$ on \eqref{eq:NDMdiagram} naturally extends to an action of $\bGmp$ on its bordered compactification.
Consider the morphisms
\[
\xymatrix@R=3ex{
\inb{(T_NM)} \ar@{^(->}[r]^-{i_\nd} & \inb{(\ND M)} & \inb\Omega \ar@{_(->}[l]_-{j_\nd} \ar[r]^-{\pnd|_\Omega} & M\,,
}
\]
In the following, when there is no risk of confusion, we will write for short $i=i_\nd$, $j=j_\nd$ and $p_\Omega=\pnd|_\Omega$.

\begin{definition}\label{def:Enu}
For $K\in\BEC M$, we set
\begin{align*}
\Enu_N(K)
&\defeq \Eopb i\Eoim {{j}}\Eopb {p_\Omega} K  \quad \in\BECcon{\inb{(T_NM)}}, \\
\Enu^\sph_N(K)
&\defeq \bl\Enu_N(K)\br^\sph  \quad \in\BEC{S_NM}.
\end{align*}
The functor $\Enu_N$ is called enhanced specialization along $N$. 
\end{definition}

With a proof similar to that of Lemma~\ref{lem:nurbbis} (or of \cite[Lemma 4.2.1]{KS90}), one has

\begin{lemma}\label{le:nuother}

For $K\in\BEC M$, one has
\[
\Enu_N(K)
\simeq \Eepb i\Eeeim j\Eepb{p_\Omega} K.
\]
\end{lemma}

Note that there is an isomorphism
\begin{equation}\label{eq:enu}
e\circ\nu_N \simeq \Enu_N\circ e,
\end{equation}
and similarly for $e$ replaced by $\epsilon$, $\epsilon^+$ or $\epsilon^-$.

Consider the morphisms
\[
\xymatrix@R=2ex@C=3em{
M &{N\;} \ar@{_(->}[l]_{i_N} \ar@{^(->}@<.5ex>[r]^-o & \inb{(T_NM)} \ar@<.5ex>[l]^(.6)\tau & \inb{(\bdot T_NM)} \ar@{_(->}[l]_u \ar@<.7ex>@/^3.5ex/[ll]^-{\bdot\tau} \ar[r]^\gamma & S_NM,
}
\]
where $o$ is the zero-section. 

\begin{lemma}\label{lem:nu}
For $K\in\BEC M$, one has the isomorphisms
\begin{itemize} 
\item[(i)]
$\Eoim \tau \Enu_N(K) \simeq \Eopb o \Enu_N(K) \simeq \Eopb i_N K$,
\item[(ii)]
$\Eeeim \tau \Enu_N(K) \simeq \Eepb o \Enu_N(K) \simeq \Eepb i_N K$,
\end{itemize}
and distinguished triangles
\begin{itemize} 
\item[(iii)]
$\Eeeim{\bdot\tau}\Eopb u\Enu_N K \to \Eepb{i_N} K \to \Eopb{i_N} K \to[+1] $,
\item[(iv)]
$\Eeeim u\Eopb \gamma\Enu^\sph_N K \to \Enu_N K \to \Eoim o \Eopb{i_N} K \to[+1]$,
\item[(v)]
$\Eoim o \Eepb{i_N} K \to \Enu_N K \to \Eoim u\Eepb \gamma\Enu^\sph_N K \to[+1]$.
\end{itemize}
\end{lemma}

\begin{proof}
(i-a)  The isomorphism $\Eoim \tau \Enu_N(K) \simeq \Eopb o \Enu_N(K)$ follows from Lemma \ref{lem:Vconic} (ii).

\medskip\noindent
(i-b) 
Let us show that the composition
\begin{equation}
\label{eq:onui}
\Eoim\tau \Enu_N(K)\to \Eoim\tau \Enu_N(\Eoim{{i_N}}\Eopb i_N K)
\isoto\Eopb i_N K 
\end{equation} 
is an isomorphism.
Since the problem is local on $N$, we may work in coordinates as in Remark~\ref{re:blowcoord}~(iii).

Recall that $\overline\Omega=\opb\snd(\R_{\geq 0})=\Omega\dunion T_NM$ is the closure of $\Omega$ in $\ND M$, and
consider the map $r\colon\overline\Omega\to M\times\R$ given by $r(v,y,s) = (sv,y,s-|v|)$. 
Then $r$ is a proper map since, in the commutative diagram
\[
\xymatrix@C=3ex@R=1em{
\overline\Omega \ar[r]^r \ar[d]^f & M\times\R \ar[d]^g \\
N\times\R_{\geq 0}\times\R \ar[r]^h & N\times\R_{\geq 0}\times\R,
}
\]
$f$ and $h$ are proper. Here, $f(v,y,s)=(y,|v|,s)$, $g(x,y,s)=(y,|x|,s)$, and $h(y,u,s)=(y,su,s-u)$.

Setting $Z=N\times\R_{\leq 0}\subset M\times\R$ and $U=( M\times\R)\setminus Z$, 
the continuous map  $r$ induces a homeomorphism $\Omega\isoto U$. Consider the commutative diagram of bordered spaces semiproper over $M$, whose two top squares are cartesian,
\[
\xymatrix@R=4ex@C=7ex{
\inb\Omega \ar@{^(->}[r]^j \ar[d]^{r_j}_\wr & \inb{\overline\Omega} \ar@/^2.5pc/[dd]^(.75){\overline p} |!{[dr];[d]}\hole \ar[d]^r & \inb{(T_NM)} \ar@{_(->}[l]_-i \ar[d]_{r_i} \ar@/^2pc/[dd]^-\tau  \\
\inb U  \ar[dr]_{q_{U}} \ar@{}[ur]|-\square \ar@{^(->}[r]^-{\tilde\jmath} & M\times\inb\R \ar@{}[ur]|-\square \ar[d]_-q & \inb Z \ar@{_(->}[l]_-{\tilde\imath} \ar[d]_{q_{Z}} \\
& M & N \ar@{_(->}[l]^{i_N}.
}
\]
Here, $q$, $q_{U}$ and $q_{Z}$ denote the first projections, and $\overline p=\pnd|_{\overline\Omega}$.
One has
\begin{align*}
\Eoim\tau \Enu_N(K)
&\simeq \Eoim\tau\Eopb i\Eoim j\Eopb j \Eopb {\overline p} K \\
&\simeq \Eoim{{q_{Z}}}\Eoim{{r_i}}\Eopb i\Eoim j\Eopb j \Eopb r\Eopb {q} K \\
&\simeq \Eoim{{q_{Z}}}\Eoim{{r_i}}\Eopb i\Eoim j \Eopb{r_j}\Eopb{\tilde\jmath}\Eopb {q} K \\
&\underset{(*)}\simeq \Eoim{{q_{Z}}}\Eopb{\tilde\imath}\Eoim r\Eoim j \Eopb{r_j}\Eopb{\tilde\jmath}\Eopb {q} K \\
&\simeq \Eoim{{q_{Z}}}\Eopb{\tilde\imath}\Eoim{\tilde\jmath}\Eoim {{r_j}} \Eopb{r_j}\Eopb{\tilde\jmath}\Eopb {q} K \\
&\simeq \Eoim{{q_{Z}}}\Eopb{\tilde\imath}\Eoim{\tilde\jmath}\Eopb{\tilde\jmath}\Eopb {q} K ,
\end{align*}
where $(*)$ follows from the properness of $r$.
Consider the commutative diagram
\[
\xymatrix@R=3ex{
\inb U \ar[d]_{q_{U}} \ar@{^(->}[r]^-{\tilde\jmath} & M\times\inb\R \ar[dl]^{q} & \inb Z \ar@{_(->}[l]_-{\tilde\imath} \ar[r]^{q_{Z}} & N \\
M \ar@{=}[r] & M \ar@{}[ur]|-\square  \ar@{_(->}[u]_-{\tilde\imath_0} & N \ar@{_(->}[l]^{i_N} \ar@{_(->}[u]_{\tilde o} \ar@{=}[ur],
}
\]
where $\tilde o(y)=(0,y,0)$ and $\tilde\imath_0(x,y) = (x,y,0)$.
Note that $\bGmp$ acts on the second component of $M\times\inb\R$, and hence also on $\inb Z$. Then
\begin{align*}
\Eoim\tau \Enu_N(K)
&\underset{(*)}\simeq \Eopb{\tilde o}\Eopb{\tilde\imath}\Eoim{\tilde\jmath}\Eopb{\tilde\jmath}\Eopb {q} K \\
&\simeq \Eopb{i_N}\Eopb{\tilde\imath_0}\Eoim{\tilde\jmath}\Eopb{\tilde\jmath}\Eopb {q} K \\
&\underset{(**)}\simeq \Eopb{i_N}\Eoim{{q}}\Eoim{\tilde\jmath}\Eopb{\tilde\jmath}\Eopb {q} K \\
&\simeq \Eopb{i_N}\Eoim{{q_{U}}}\Eopb {q_{U}} K,
\end{align*}
where $(*)$ holds since $\Eopb{\tilde\imath}\Eoim{\tilde\jmath}\Eopb{\tilde\jmath}\Eopb {q} K$ is $\bGmp$-conic, and  $(**)$ holds since $\Eoim{\tilde\jmath}\Eopb{\tilde\jmath}\Eopb {q} K$ is $\bGmp$-conic. Then, one has
\begin{align*}
\Eoim\tau \Enu_N(K)
&\simeq \Eopb{i_N}\Eoim{{q_{U}}}\cihom(\epsilon(\field_U),\Eepb {(q_{U})} K[-1]) \\
&\simeq \Eopb{i_N}\cihom(\Eeeim{{q_{U}}}\epsilon(\field_U),K[-1]) \\
&\underset{(*)}\simeq \Eopb{i_N}\cihom(\epsilon(\reim{{q_{U}}}\field_U),K[-1]) \\
&\underset{(**)}\simeq \Eopb{i_N}\cihom(\epsilon(\field_M[-1]),K[-1]) \\
&\simeq \Eopb{i_N}K,
\end{align*}
where $(*)$ holds since $q_{U}$ is semiproper, and $(**)$ is due to the fact that the fibers of $q_{U}$ are homeomorphic to $\R$.
This gives \eqref{eq:onui}.

\medskip\noindent
(ii) has a similar proof to (i).

\medskip\noindent
(iii) follows from (i) and (ii), using Lemma~\ref{lem:Vconic} (iv).

\medskip\noindent
(iv) Consider the distinguished triangle
\[
\Eeeim u \Eopb u \Enu_N(K) \to \Enu_N(K)\to \Eeeim o \Eopb o \Enu_N(K) \to [+1].
\]
Then the statement follows from (i) and  Lemma~\ref{lem:Vconic} (i).

\medskip\noindent
(v) has a proof similar to that of (iv), using the distinguished triangle
\[
\Eoim o \Eepb o \Enu_N(K) \to \Enu_N(K) \to \Eoim u \Eepb u \Enu_N(K) \to [+1].
\]
\end{proof}

Recall the definition of the normal cone in \eqref{eq:CN}.
Here is an analogue of \cite[Exercise IV.2]{KS90}

\begin{lemma}\label{le:EnuMU}
Let $S\subset M$ be a closed subset. Then $\Enu_N$ induces a functor $\Enu_{N|M\setminus S}$ (see \eqref{eq:EnuU} below) entering the quasi-commutative diagram
\[
\xymatrix@C=4em{
\BEC M \ar[d] \ar[r]^-{\Enu_N}\ar[d]^{\Eopb j_S} & \BEC{\inb{(T_NM)}} \ar[d]^{\Eopb j_C} \\
\BEC {\inb{(M\setminus S)}}\ar[r]^-{\Enu_{N|M\setminus S}} & \BEC {\inb{(T_NM\setminus C_NS)}} ,
}
\] 
where $j_S\colon \inb{(M\setminus S)} \to M$ and $j_C\colon \inb{(T_NM\setminus C_NS)} \to \inb{(T_NM)}$ are the open embeddings.
\end{lemma}

\begin{proof}
Consider the commutative diagram with cartesian squares
\[
\xymatrix@R=3ex@C=3ex{
M & \inb\Omega \ar[l]_-{p_\Omega} \ar[r]^-{i} & \inb{(\ND[N]M)} & \inb{(T_NM)} \ar[l]_-{ j} \\
\inb{(M\setminus S)} \ar@{^(->}[u]^{j_S} \ar@{}[ur]|-\square & \inb{(\Omega\setminus \opb p_\Omega S)} \ar[l]_-{p_\Omega'} \ar[r]^-{i'} \ar@{^(->}[u]^{j_1} \ar@{}[ur]|-\square & \inb{(\ND[N]M\setminus \overline{\opb p_\Omega S})} \ar@{^(->}[u]^{j_2} \ar@{}[ur]|-\square & \inb{(T_NM\setminus C_NS)} \ar[l]_-{ j'} \ar@{^(->}[u] \ar@{^(->}[u]^{j_C}  ,
}
\]
where all the vertical arrows are open embeddings. Then 
\begin{align*}
\Eopb j_C \Enu_N (K) 
&= \Eopb j_C \Eoim{ j} \Eoim{{i}} \Eopb{p_\Omega} K \\
&\simeq \Eoim{ j'} \Eoim{i'} \Eopb{p_\Omega'} \Eopb{j_S} K .
\end{align*}

For $K'\in\BEC {\inb{(M\setminus S)}}$, set
\begin{equation}\label{eq:EnuU}
\Enu_{N|M\setminus S}(K') \defeq \Eoim{ j'} \Eoim{{i'}} \Eopb{p_\Omega'} K'.
\end{equation}
Then the statement is clear.
\end{proof}

Here is an analogue of \cite[Exercise IV.5]{KS90}

\begin{lemma}\label{lem:nu_Nid}
Let $\tau\colon V\to N$ be a vector bundle, and denote by $o\colon N\to V$ the embedding of the zero-section. 
For $K\in\BECcon {\inb V}$, one has
\[
\Enu_N(K) \simeq K,
\]
where we use the identifications $N\simeq o(N)\subset V$ and $T_NV\simeq V$.
\end{lemma}

\begin{proof}
One has $\inb{(\ND V)} \simeq V\times \bR$, with $\snd(x,s)=s$ and $\pnd(x,s)=sx$.
Hence $\inb\Omega\simeq V\times\bGmp$, and $p_\Omega=\mu$, where $\mu$ is the $\bGmp$-action. Then,
\[
\Eopb{p_\Omega} K = \Eopb\mu K \simeq K \cetens e_\bR(\field_{\R_{>0}}),
\]
where the last isomorphism is due to the fact that $K$ is $\bGmp$-conic.
Recalling Definition~\ref{def:Enu}, the statement easily follows.
\end{proof}

Let $f\colon M_1\to M_2$ be a morphism of real analytic manifolds, let $N_i\subset M_i$ ($i=1,2$) be closed submanifolds, and assume $f(N_1)\subset N_2$.
Consider the associated morphism, given by the composition
\[
T_{N_1}f\colon \inb{(T_{N_1}M_1)} \to[f']  N_1\times_{N_2}\inb{(T_{N_2}M_2)} \to  \inb{(T_{N_2}M_2)}.
\]
The enhanced specialization functor satisfies the analogous functorial properties as those in Propositions 4.2.4, 4.2.5 and 4.2.6 of \cite{KS90}. The proofs in loc.\ cit.\ immediately extend to the enhanced framework. 

For example, if $f$ and $f|_{N_1}\colon N_1\to N_2$ are smooth, one has
\begin{equation}\label{eq:Enuopb}
\Eopb{(T_{N_1}f)}\circ\Enu_{N_2} \simeq \Enu_{N_1}\circ\Eopb f . 
\end{equation}

\subsection{Blow-up transform}\label{sse:blow}

Let $M$ be a real analytic manifold, and $N\subset M$ a closed submanifold.
With notations as in \eqref{eq:RBMdiagram}, consider the real oriented blowup $\RB M$ and the commutative diagram of bordered spaces
\begin{equation}
\label{eq:blowMN}
\xymatrix@R=3ex{
S_NM \ar@{^(->}[r]^-{i_\rb} \ar[d]_\sigma & \RB M \ar[d]_{\prb}
& \inb{(M\setminus N)} \ar@{_(->}[dl]^-{j_N} \ar@{_(->}[l]_-{j_\rb} \\
N \ar@{^(->}[r]^-{i_N} \ar@{}[ur]|-\square & M\,.
}
\end{equation}
Note that $\inb{(M\setminus N)} \simeq \inb{(\RB M\setminus S_NM)}$.
In the following, when there is no risk of confusion, we will write for short $i=i_\rb$, $j=j_\rb$ and $p=\prb$.

\begin{definition}
For $K\in\BEC M$, consider the object
\[
\Enu^\rb_N(K)
\defeq \Eopb i\Eoim j\Eopb j_N K \quad\in \BEC{S_NM}.
\]
We denote by
\[
\nu^\rb_N\colon
\BDC(\field_M) \to \BDC(\field_{S_NM})
\]
the analogous functor for sheaves.
\end{definition}

Note that, by definition, $\Enu^\rb_N$ factors through a functor, that we denote by the same name,
\[
\Enu^\rb_N\colon\BEC{\inb{(M\setminus N)}} \to \BEC {S_NM}.
\]
Note also that one has
\begin{equation}\label{eq:enurd}
e\circ\nu_N^\rb \simeq \Enu_N^\rb\circ e,
\end{equation}
and similarly for $e$ replaced by $e\circ\iota$, $\epsilon$, $\epsilon^+$ or $\epsilon^-$.

\begin{lemma}\label{lem:nurbbis}
For $K\in\BEC M$, one has
\[
\Enu^\rb_N(K)
\simeq \Eepb i\Eeeim j\Eepb j_N K[1].
\]
\end{lemma}

\begin{proof}
For $L\in\BEC{\RB M}$, there is a distinguished triangle
\[
\Eoim i\Eepb i L \to
L \to
\Eoim j\Eepb j L \to[+1].
\]
When $L=\Eeeim j\Eopb j_N K$, the above distinguished triangle reads
\[
\Eoim i\Eepb i\Eeeim j\Eepb j_N K \to \Eeeim j\Eopb j_N K \to \Eoim j\Eopb j_N K \to[+1].
\]
The statement follows by applying $\Eopb i$, and noticing that $\Eopb i\Eeeim j\simeq0$.
\end{proof}

\begin{lemma}\label{lem:nuS}
For $K\in\BEC M$, one has
\[
\Enu^\sph_N(K) \simeq \Enu^\rb_N(K).
\]
\end{lemma}

\begin{proof}
Considering the morphisms
\[
\xymatrix@R=1ex{
\inb{(T_NM)} & \inb{(\bdot T_NM)} \ar@{_(->}[l]_-u \ar[r]^-\gamma & S_NM,
}
\]
it is equivalent to prove
\[
\Eopb u\Enu_N(K) \simeq \Eopb\gamma\Enu^\rb_N(K).
\]
Consider the commutative diagram, extending \eqref{eq:dNDMdiagram}, whose squares are Cartesian with smooth vertical arrows
\[
\xymatrix@R=2ex{
& \inb{(\ND M)} \\
\inb{(T_NM)} \ar@{^(->}[ur]^{i_\nd} \ar@{^(->}[r] & \inb{\overline\Omega} \ar@{^(->}[u] & \inb\Omega \ar@{_(->}[ul]_ {j_\nd} \ar@{_(->}[l] \ar[dr]^-{p_\Omega} \\
\inb{(\bdot T_NM)}\ar@{}[ur]|-\square \ar@{^(->}[r] \ar@{^(->}[u]_u \ar[d]^\gamma & \inb{\widetilde\Omega} \ar@{}[ur]|-\square \ar@{^(->}[u] \ar[d] & \inb{(\widetilde\Omega\cap\Omega)} \ar@{_(->}[l]  \ar@{^(->}[u] \ar[d] & M \,. \\
S_NM \ar@{}[ur]|-\square \ar@{^(->}[r]^{i_\rb} & \RB M \ar@{}[ur]|-\square & \inb{(M\setminus N)}  \ar@{_(->}[l]_-{j_\rb} \ar[ur]_-{j_N}
}
\]
We have to prove
\begin{equation}\label{eq:sphrbtemp}
\Eopb u\Eopb{i_\nd}\Eoim {{j_\nd}}\Eopb {p_\Omega} K 
\simeq \Eopb\gamma \Eopb{i_\rb}\Eoim{{j_\rb}}\Eopb{j_N}  K.
\end{equation}
This is obtained by chasing the above diagram.
\end{proof}

\section{Fourier-Sato transform and microlocalization}\label{se:mu}

We recall here the natural enhancement of the Fourier-Sato transform from \cite[\S3]{Tam18} (see also \cite{DAg14} and \cite{KS16L}), referring to \cite[\S3.7]{KS90} for the classical case. We then define the natural enhancement of Sato's microlocalization, for which we refer to \cite[\S4.3]{KS90}. 

\subsection{Kernels}
Let $\bdx$ and $\bdy$ be bordered spaces.
A kernel from $\bdx$ to $\bdy$ is a triple $(p,q,C)$, where $p$ and $q$ are morphisms of bordered spaces
\[
\xymatrix@R=1ex{
& \bds \ar[dl]_p \ar[dr]^q \\
\bdx && \bdy,
}
\]
and $C\in\BEC{\bds}$. To such a kernel one associates the functors
\[
\Phi_{(p,q,C)},\Psi_{(p,q,C)}\colon\BEC\bdx\to\BEC\bdy,
\]
defined by
\begin{align*}
\Phi_{(p,q,C)}(K) &\defeq \Eeeim q(C\ctens\Eopb p K), \\
\Psi_{(p,q,C)}(K) &\defeq \Eoim q\cihom(C,\Eepb p K).
\end{align*}

Given a commutative diagram
\[
\xymatrix@R=1ex@C=3em{
& \bds \ar[dddl]_p \ar[dddr]^q \ar[dd]^r \\ & \\
& \bds' \ar[dl]^{p'} \ar[dr]_{q'} \\
\bdx && \bdy,
}
\]
one has
\begin{equation}\label{eq:Cpush}
\Phi_{(p,q,C)} \simeq \Phi_{(p',q',\reeim r C)}, \quad
\Psi_{(p,q,C)} \simeq \Psi_{(p',q',\reeim r C)}.
\end{equation}

If there is no fear of confusion, we will write for short
\[
C = (p,q,C), \quad C^\XY = (q_1,q_2,\reim{(p,q)} C), \quad C^r = (q,p,C),
\]
where $q_1$ and $q_2$ are the projections from $\bdx\times\bdy$ to $\bdx$ and $\bdy$, respectively,and $(p,q)\cl \bds\to \bdx\times\bdy$ is the morphism induced by $p$ and $q$.
Then, \eqref{eq:Cpush} implies
\[
\Phi_C \simeq \Phi_{C^\XY}, \quad
\Psi_C \simeq \Psi_{C^\XY},
\]
and the kernel $C^r$ from $\bdy$ to $\bdx$ gives functors
\[
\Phi_{C^r},\Psi_{C^r}\colon\BEC\bdy\to\BEC\bdx.
\]
Note that $\Phi_C$ is left adjoint to $\Psi_{C^r}$ (and $\Psi_C$ is right adjoint to $\Phi_{C^r}$). Note also that for $K\in\BEC\bX$ one has
\begin{equation}\label{eq:Phidual}
\Edual_\bY\Phi_C(K) \simeq \Psi_C(\Edual_\bX K).
\end{equation}

Note that, if $C\in\BECpm{\bds}$, then $\Phi_C$ and $\Psi_C$ take value in $\BECpm\bdy$.
In this case, we set
\[
\Phi^\pm_C\defeq\Phi_C\vert_{\BECpm\bdx}, \quad
\Psi^\pm_C\defeq\Psi_C\vert_{\BECpm\bdx},
\]
so that we have functors
\[
\Phi^\pm_C, \Psi^\pm_C\colon\BECpm\bdy\to\BECpm\bdx.
\]

For $\ast\in\st{\emptyset,+,-}$,
consider the kernel $\mathbf{1}^\ast_\bdx  \defeq (q_1,q_2,\epsilon^\ast(\field_{\Delta_X}))$, where $X=\oo\bdx$ and $\Delta_X\subset X\times X$ is the diagonal.
Note that one has $\mathbf{1}^\ast_\bdx \simeq (\id_\bX,\id_\bX,\epsilon^\ast(\field_X))^\XY$, so that in particular
\[
\Psi^\ast_{\mathbf{1}^{\ast}_\bdx}\simeq \id_{\BECast\bdx} \simeq
\Phi^\ast_{\mathbf{1}^{\ast}_\bdx}.
\]

Given another bordered space $\bdz$, and a kernel $D=(\bdt\to[r]\bdy,\bdt\to[s]\bdz,D)$ from $\bdy$ to $\bdz$, consider the diagram with cartesian square
\[
\xymatrix@R=1ex@C=3em{
&& \bds\times_\bdy\bdt \ar[dl]_{r'} \ar[dr]^{q'} \\
& \bds  \ar[dl]_p \ar[dr]^q \ar@{}[rr]|-\square && \bdt  \ar[dl]_r \ar[dr]^s \\
\bdx && \bdy && \bdz.
}
\]
Setting
\[
C \ccomp D \defeq \Eopb{r'}C \ctens \Eopb{q'}D \in \BEC{\bds\times_\bdy\bdt},
\]
one gets a kernel $C \ccomp D = (p\circ r', s\circ q', C \ccomp D)$ from $\bdx$ to $\bdz$ such that
\[
\Phi_D \circ \Phi_C \simeq \Phi_{C \ccomp D}, \quad
\Psi_D \circ \Psi_C \simeq \Psi_{C \ccomp D}.
\]

Let $\ast\in\st{\emptyset,+,-}$, $C\in\BECast\bds$ and $D\in\BECast\bdt$.
Assume that $\bdz=\bdx$ and that
\[
(C\ccomp D)^\XY \simeq \mathbf{1}^\ast_\bdx, \quad
(D\ccomp C)^\XY \simeq \mathbf{1}^\ast_\bdy.
\]
Then, the functors $\Phi_C^\ast$ and $\Phi_D^\ast$ (resp.\ $\Psi^\ast_C$ and $\Psi_D^\ast$) are equivalences of categories quasi-inverse to each other. Moreover, by uniqueness of the adjoint, one has
\begin{equation}\label{eq:PhiPsi}
\Phi^\ast_C \simeq \Psi^\ast_{D^r}, \quad \Psi^\ast_C \simeq \Phi^\ast_{D^r}.
\end{equation}

\begin{lemma}\label{lem:cartcomp}
Consider a commutative diagram of bordered spaces with cartesian squares
\[
\xymatrix@R=3ex{
\bdx' \ar[d]_f & \bds' \ar[d]_h \ar[l]_-{p'} \ar[r]^-{q'} & \bdy' \ar[d]^g \\
\bdx \ar@{}[ur]|-\square & \bds \ar[l]^-p \ar[r]_-q \ar@{}[ur]|-\square & \bdy.
}
\]
Let $C\in\BEC\bds$, and set 
\[
C'\defeq \Eopb h C \in \BEC{\bds'}.
\]
Consider $C=(p,q,C)$ and $C'=(p',q',C')$ as kernels.
Then, one has
\begin{align*}
\Phi_C\circ\Eeeim f &\simeq \Eeeim g\circ\Phi_{C'}, &
\Psi_C\circ\Eoim f  &\simeq \Eoim g\circ\Psi_{C'}, \\
\Phi_{C'}\circ\Eopb f  &\simeq \Eopb g\circ\Phi_C, &
\Psi_{C'}\circ\Eepb f  &\simeq \Eepb g\circ\Psi_C.
\end{align*}
\end{lemma}

\subsection{Enhanced Fourier-Sato transform}

Let $\bbM$ be a bordered space, and $U\subset \bbM$ an open subset. For $\varphi\colon U\to\R$ a continuous function, consider the object of $\BECp\bbM$
\[
\ex_U^\varphi \seteq \quot_\bbM\field_{\st{t+\varphi(x)\geq 0}},
\]
where we write for short 
\[
\st{t+\varphi(x)\geq 0} = \st{(x,t)\in U\times\R\semicolon t+\varphi(x)\geq 0}.
\]

Let $\tau\colon V\to N$ be a real vector bundle and $\varpi\colon V^*\to N$ its dual bundle.
Denote by $\inb V$ and $\inb{V^*}$ their bordered compactifications, consider the projections
\[
\xymatrix{
\inb V & \inb V\times_N \inb{V^*} \ar[l]_-p \ar[r]^-q & \inb{V^*},
}
\]
and let $\langle \cdot,\cdot\rangle\colon V\times_N V^* \to \R$ denote the pairing.

\begin{notation}\label{not:Fou}
Let $K\in\BECp \Vi$.
\begin{itemize}
\item[(i)]
The Fourier-Sato transforms (see \cite[\S3.7]{KS90} for the case of sheaves) are defined by
\[
K^\wedge \defeq\Phi^+_\Fou(K), \quad K^\vee \defeq\Phi^+_\Foua(K),
\]
for $\Fou,\ \Foua\in\BECp{\inb V\times_N\inb{V^*}}$ given by
\[
\Fou \defeq \epsilon^+\bl\field_{\st{\langle x,y\rangle\leq0}}\br, \quad
\Foua \defeq \epsilon^+\bl\field_{\st{\langle x,y\rangle\geq0}}
\tens\opb p\omega_{V/N}\br.
\]
Note that the kernels $\Fou$ and $\Foua$ are $\bGmp$-bi-conic for the actions $c\cdot(x,y) = (cx,y)$ and $d\cdot (x,y) = (x,dy)$.
Hence, the Fourier-Sato transforms take values in $\BECcon\Wi\cap\BECp\Wi$.
\item[(ii)]
The enhanced Fourier-Sato transforms (see \cite[\S3.1.3]{Tam18} for the case of enhanced sheaves) are defined by
\[
\lap K \defeq\Phi^+_\Lap(K), \quad \lapa K \defeq\Phi^+_\Lapa(K),
\]
for $\Lap,\ \Lapa\in\BECp{\inb V\times_N\inb{V^*}}$ given by
\[
\Lap \defeq \ex_{V\times_N V^*}^{-\langle x,y\rangle}, \quad
 \Lapa \defeq  \ex_{V\times_N V^*}^{\langle x,y\rangle}\tens\opb\pi\opb p\omega_{V/N}.
\]
Note that the kernels $\Lap$ and $\Lapa$ are $\bGmp$-conic for the action $c\cdot(x,y) = (cx,c^{-1}y)$.
Hence the enhanced Fourier-Sato transforms sends
$\BECcon \Vi$ to $\BECcon \Wi\cap\BECp \Wi$.
\end{itemize}
\end{notation}

It is shown in \cite{Tam18} (see also \cite{KS16L}) that one has
\begin{equation}\label{eq:LapLapr}
\bl\Lap \ccomp \Lapa^r\br^\XY \simeq \mathbf{1}^+_{\Vi}, \quad
\bl\Lapa^r \ccomp \Lap\br^\XY \simeq \mathbf{1}^+_{\Wi}.
\end{equation}
It follows that $\lap(\cdot)$ and $\lapar(\cdot)$ are quasi-inverse to each other and that, by \eqref{eq:PhiPsi}, 
\begin{equation}\label{eq:FouPhiPsi}
\lap K
\simeq \Psi^+_\Lapa(K), \quad
\lapa K
\simeq \Psi^+_\Lap(K).
\end{equation}

Note that one has
\begin{align}
\label{eq:eFou}
e\bl K^\wedge\br &\simeq \bl e(K)\br^\wedge, &
e\bl K^\vee\br &\simeq \bl e(K)\br^\vee, 
\end{align}
and the same for $e$ replaced by $\epsilon^+$.

The following result was obtained in \cite{DAg14,KS16L} for conic sheaves, and we generalize it to enhanced ind-sheaves.

\begin{proposition}\label{pro:Foucon}
For $K\in\BECcon \Vi\cap\BECp \Vi$, one has
\[
\lap K \simeq   K^\wedge, \quad
\lapa K \simeq  K^\vee.
\]
\end{proposition}

\begin{proof}
We will adapt the proof of \cite[Theorem 5.7]{KS16L}. Since the arguments are similar, we will only treat the first isomorphism.

One has
\[
\Lap = \quot\,\field_{\st{t-\langle x,y\rangle\geq0}} ,\quad 
\Fou = \quot\,\field_{\st{t\geq0\geq\langle x,y\rangle}}.
\]
The inclusion $\st{t-\langle x,y\rangle\geq0} \supset\st{t\geq0\geq\langle x,y\rangle}$ induces a distinguished triangle
\[
\quot\,\field_{\st{t\geq\langle x,y\rangle>0}}\dsum\quot\,\field_{\st{0>t\geq\langle x,y\rangle}}
\to \quot\,\field_{\st{t-\langle x,y\rangle\geq0}} \to \quot\,\field_{\st{t\geq0\geq\langle x,y\rangle}} \To[+1].
\]
We are thus left to prove
\[
\Phi_{\quot\,\field_{\st{t\geq\langle x,y\rangle>0}}}(K) \simeq 0
\simeq \Phi_{\quot\,\field_{\st{0>t\geq\langle x,y\rangle}}}(K).
\]
Since the arguments are similar, we will only treat the first isomorphism.

Consider the morphisms
\[
\xymatrix@R=3ex{
\Vi & \Vi \times \Wi \ar[l]_-p \ar[d]^-q \ar[r]^h & \Wi\times\R_\infty \ar[dl]_r & \Wi\times\inb{(\R_{>0})} \ar@{_(->}[l]_-j \ar[dll]^s \\
& \Wi,
}
\]
where $h(x,y)=(y,\langle x,y\rangle)$, $j$ is the embedding, and $p$, $q$, $r$, $s$ are the projections.
Then, denoting by $\la$ the coordinate of $\bR$,
\begin{align*}
\Phi_{\quot\,\field_{\st{t\geq\langle x,y\rangle>0}}}(K) 
&= \Eeeim q \bl \quot\,\field_{\st{t\geq\langle x,y\rangle>0}} \ctens \Eopb p K \br \\
&\simeq \Eeeim r \Eeeim h \bl  \Eopb h\quot\,\field_{\st{t\geq\lambda>0}} \ctens \Eopb p K \br \\
&\simeq \Eeeim r \bl \quot\,\field_{\st{t\geq\lambda>0}} \ctens \Eeeim h \Eopb p K \br \\
&\simeq \Eeeim r \bl \epsilon(\field_{\st{\lambda>0}}) \ctens 
\quot\field_{\st{t\geq\lambda}} \ctens \Eeeim h \Eopb p K  \\
&\simeq \Eeeim r \Eeeim j \Eopb j ( \quot\,\field_{\st{t\geq\lambda}} \ctens \Eeeim h \Eopb p K) \\
&\simeq \Eeeim s \bl \quot\,\field_{\st{t\geq\lambda}} \ctens \Eopb j \Eeeim h \Eopb p K \br.
\end{align*}
Since $\Eopb j \Eeeim h \Eopb p K$ is $\bGmp$-conic for the action $c\cdot(y,\lambda)=(y,c\lambda)$, we have $\Eopb j \Eeeim h \Eopb p K \simeq \Eopb s H$ for some $H\in\BEC\Wi$. Hence
\begin{align*}
\Phi_{\quot\,\field_{\st{t\geq\langle x,y\rangle>0}}}(K) 
&\simeq \Eeeim s \bl \quot\,\field_{\st{t\geq\lambda}} \ctens \Eopb s H \br \\
&\simeq \Eeeim s \bl\quot\,\field_{\st{t\geq\lambda}}\br  \ctens  H  \\
&\underset{(*)}\simeq \quot\bl\reim {{s_\R}} \field_{\st{t\geq\lambda}}\br  \ctens  H \simeq 0,
\end{align*}
where $(*)$ follows since $s$ is semiproper. (Recall that $s_\R=s\times\id_\R$.)
\end{proof}

By Lemma~\ref{lem:cartcomp}, we obtain the following analogue of \cite[Proposition 3.7.13]{KS90}.

\Lemma\label{lem:lapinv}
Let $V\to N$ be a vector bundle and let $f\cl N'\to N$ be a morphism.
Set $V'=V\times_NN'$ and $V^{\prime\,*}=V^*\times_NN'$,
and let $g\cl V'\to V$ and $h\cl V^{\prime\,*}\to V^*$ be the induced morphism.
Then
\bnum
\item
For any $K\in \BECp \Vi$, we have
$$\Eopb{h}(\lap K)\simeq\lap\bl\Eopb{g}K\br\qtq
\Eepb{h}(\lap K)\simeq\lap\bl\Eepb{g}K\br.$$
\item
For any $K'\in \BECp{(V')_\infty}$, we have
$$\Eoim{h}(\lap K')\simeq\lap\bl\Eoim{g}K'\br\qtq
\Eeeim{h}(\lap K')\simeq\lap\bl\Eeeim{g}K'\br.$$
\ee
\enlemma

The enhanced Fourier functor satisfies also other functorial properties, as those in Propositions 3.7.14 and 3.7.15 of \cite{KS90}.  The first one was already pointed out in \cite[\S5.2]{KS16L}, and the second one easily follows from
\[
\ex^{-\langle x_1,y_1\rangle}_{V_1\times_N V_1^*} \cetens[N] \ex^{-\langle x_2,y_2\rangle}_{V_2\times_N V_2^*} \simeq \ex^{-\langle (x_1,x_2),(y_1,y_2)\rangle}_{(V_1\times_N V_2)\times_N (V_1\times_N V_2)^*}.
\]

\subsection{Enhanced microlocalization}

As in \S~\ref{sse:spec}, let $M$ be a real analytic manifold, and $N\subset M$ a closed submanifold.

\begin{definition}
For $K\in\BECp M$, we set
\begin{align*}
\Emu_N(K)
&\defeq \lap\bl\Enu_N(K)\br \quad \\
&\simeq \bl\Enu_N(K)\br^\wedge \quad \in\BECcon{\inb{(T^*_NM)}}\cap\BECp{\inb{(T^*_NM)}}, \\
\Emu^\sph_N(K)
&\defeq \bl\Emu_N(K)\br^\sph  \quad \in\BECp{S^*_NM},
\end{align*}
where the isomorphism follows from Proposition~\ref{pro:Foucon}, since $\Enu_N(K)$ is $\bGmp$-conic.
The functor $\Emu_N$ is called enhanced microlocalization along $N$. 
\end{definition}

Note that one has
\begin{equation}\label{eq:emu}
e\circ\mu_N \simeq \Emu_N\circ e,
\end{equation}
and similarly for $e$ replaced by $\epsilon^+$.

The enhanced microlocalization functor satisfies the analogous functorial properties as those in Propositions 4.3.4, 4.3.5 and 4.3.6 of \cite{KS90}. The proofs in loc.\ cit.\ immediately extend to the enhanced framework.

\section{Specialization at $\infty$ on  vector bundles}\label{se:smash}

On a vector bundle $\tau\colon V\to N$, we construct an enhancement of the so-called \emph{smash functor} from \cite[\S6.1]{DHMS18}, 
which is related to ``specialization at $\infty$'', and we compute its enhanced Fourier-Sato transform.

\subsection{Smash functor}\label{sse:smash}

Let $\tau\colon V\to N$ be a vector bundle, and consider the morphisms of bordered vector bundles over $N$
\[
\xymatrix{
\Vi & \ar[l]_-{\psm} \Vi\times\inb{(\R_{>0})} \ar@{^(->}[r]^-{j_\sm} & \Vi\times\inb\R & \Vi \ar@{_(->}[l]_-{i_\sm},
}
\]
where $\psm(x,s)=s^{-1}x$, $ i_\sm(x)=(x,0)$, and $j_\sm$ is the open embedding.
In the rest of this section we will write for short $p$, $i$ and $j$ instead of $\psm$, $i_\sm$ and $j_\sm$, respectively, if there is no fear of confusion.

Note that $p$, $i$, $j$ are $\bGmp$-equivariant
with respect to the ordinary actions of $\bGmp$ on $\Vi$ and $\bR$,
except the trivial action on the leftmost $\Vi$.

\begin{definition}
For $K\in\BEC\Vi$, set
\[
\Esm_V(K) \defeq \Eopb i \Eoim j \Eopb p K\quad \in\BECcon\Vi.
\]
This is called the enhanced smash functor.
\end{definition}

With a proof similar to that of Lemma~\ref{lem:nurbbis}, one has

\begin{lemma}
With the above notations, one has
\[
\Esm_V(K)
\simeq \Eepb i \Eeeim j \Eepb p K.
\]
\end{lemma}

\Lemma Let $o\cl N\to V$ be the zero section.
Then for $K\in\BEC\Vi$, one has
$$\Eopb o \Esm_V(K)\simeq \Eoim \tau K.$$
\enlemma
\Proof Let $\tilde\tau\cl \Vi\times\bR\to N$ be the projection,
and $\tilde o\cl N\to \Vi\times\bR$ the zero section.
  one has
\eqn
\Eopb o \Esm_V(K)
&&\simeq\Eopb o\Eopb i \Eoim j \Eopb p K\\
&&\simeq\Eopb{\tilde o} \Eoim j \Eopb p K\\
&&\underset{(*)}{\simeq}\Eoim{\tilde \tau} \Eoim j \Eopb p K\\[0ex]
&&\simeq\Eoim{\tau} \Eoim p \Eopb p K\\
&&\simeq\Eoim{\tau}K.
\eneqn
Here $(*)$ follows from $ \Eoim j \Eopb p K\in\BECcon{\Vi\times\bR}$.
\QED

Consider the vector bundle $W\defeq\R\times V\to N$, and let $\bdot W \defeq W\setminus(\st0\times N)$ be the complement of the zero section.
Recall from \S\ref{sse:convect} that the fiberwise sphere compactification%
\footnote{Here we choose a different compactification 
from the one in \cite[\S B.2]{DHMS18}.} of $\tau$  is $\SV\seteq\bdot W/\Gmp$,
We denote by $q\colon \bdot W\to \SV$ the quotient map, and set $[u,x] \defeq q(u,x)$.

There is a natural decomposition
\begin{equation}\label{eq:SVHV}
\SV = V^+ \sqcup H \sqcup V^-,
\end{equation}
corresponding to $u>0$, $u=0$ and $u<0$, respectively.
Note that the fibers of $H\to N$ are great spheres of codimension one in the fibers of $\SV\to N$.
Note also that there are natural identifications  $\iota^\pm\colon V\isoto V^\pm$, $x\mapsto[\pm1,x]$.
Set $N^\pm=\iota^\pm(N)\subset V^\pm$, and $N^\pm_0=N^\pm\times\st0\subset \SV\times\R$.

In order to compute the functor $\Enu_H$, let us describe the normal deformation $\ND[H]\SV$, and the bordered compactification $\inb{(\ND[H] \SV)}$ of $\pnd\colon\ND[H]\SV\to \SV$.

\begin{lemma}\label{le:ddotS}
With notations as above,
\begin{itemize}
\item[(i)]
one has
$\ND[H] \SV \simeq (\SV\times\R)\setminus(N_0^+\union N_0^-)$,
with
$\pnd([\tu,\tx],s) = [s\tu,\tx]$ and $\snd([\tu,\tx],s) = s$ \ro see Figure~\ref{fig:smash}\rf.
The $\Gm$-action on $\ND[H] \SV$ is given by $c\cdot([\tu,\tx],s)=(([c\tu,\tx]),c^{-1}s)$.
\item[(ii)]
One has $\inb{(\ND[H]\SV)} \simeq \bl (\SV\times\R)\setminus(N_0^+\union N_0^-), \SV\times\overline\R \br$.
\end{itemize}
\end{lemma}

\begin{figure}\label{fig:smash}
\begin{tikzpicture}[baseline={([yshift=-.5ex]current bounding box.center)},
                    scale=1.3]
\def\yax{.4}
\draw[->,thin] (3,-.5) -- (3,2) node [below left] {$s$} ;
\filldraw (3,0) node[left] {$0$} circle (1pt) ;
\filldraw (3,1) node[left] {$1$} circle (1pt) ;
\draw (3.5,.5) node {$\R$} ;
\draw (-2,-2.5) node {$\SV$} ;
\draw (-2,.5) node {$\ND[H]\SV$} ;
\draw[->] (0,-\yax)++ (0,-.7) -- (0,-1.9) node [midway, right] {$\pnd$} ;
\draw[->] (1.4,.5) -- (2.6,.5) node [midway, above] {$\snd$} ;
\draw[red,name path=hyp] 
  (-100:1 and \yax)
    .. controls ($(-100:1 and \yax)+(.5,.5)$) and (1,1.2) .. 
  (1,2.5) ; 
\begin{scope}
    \clip (0,-.5) circle (1 and \yax) ;
\draw[red] 
  let \p1=(-100:1 and \yax) in (\p1)
    .. controls ($(\p1)+(.5,-.5)$) and ($(1,-1.2)+(0,2\y1)$) .. 
  ($(1,-2.5)+(0,2\y1)$) ; 
\end{scope}
\filldraw[white] (0,2) rectangle (1,2.5) ;
\draw[red] (0,-.5)++ (-100:1 and \yax) -- ++(0,2.5) ;
\filldraw[very thin,fill=white, draw=black] (-1,2) arc (-180:180:1 and \yax) ;
\draw[very thin,name path=circ] (-1,1) arc (-180:0:1 and \yax) ;
\draw[very thin,densely dotted] (-1,1) arc (-180:-360:1 and \yax) ;
\filldraw[red,name intersections={of=hyp and circ}] (intersection-1) circle (1pt) node[above left,black] {$b_1$} ;
\filldraw[red] (-100:1 and \yax)++ (0,1) circle (1pt) node[above left,black] {$N^+_1$} ;
\draw[thick] (-1,0) arc (-180:0:1 and \yax) ;
\draw[thick,densely dotted] (-1,0) arc (-180:-360:1 and \yax) ;
\filldraw[fill=white, draw=black] (-100:1 and \yax) node[above left]{$N^+_0$} circle (1pt);
\filldraw[fill=white, draw=black] (80:1 and \yax) circle (1pt);
\draw (-60:1 and \yax) node[below]{$T^+_H\SV$} ;
\draw (-130:1 and \yax) node[below]{$T^+_H\SV$} ;
\draw[very thin,] (-1,-.5) arc (-180:0:1 and \yax) ;
\draw[very thin,densely dotted] (-1,-.5) arc (-180:-360:1 and \yax) ;
\draw[thick] (-1,-.5) -- (-1,2) ;
\draw[thick] (1,-.5) -- (1,2) ;
\filldraw (1,0) circle (1pt);
\filldraw (-1,0) circle (1pt) ;
\draw[thick] (-1,-2.5) arc (-180:0:1 and \yax) ;
\draw[thick,densely dotted] (-1,-2.5) arc (-180:-360:1 and \yax) ;
\filldraw[red,name intersections={of=hyp and circ}] ($(intersection-1) + (0,-3.5)$) circle (1pt) node[below right,black] {$b$};
\draw[yshift=-2.5cm] (-120:1 and \yax) node[below left]{$V^+=V$} ;
\draw[yshift=-2.5cm] (125:1 and \yax) node[above]{$V^-$} ;
\filldraw[red,yshift=-2.5cm] (-100:1 and \yax) node[below,black]{$N^+$} circle (1pt);
\filldraw[yshift=-2.5cm] (80:1 and \yax) node[above right]{$N^-$} circle (1pt) ;
\filldraw[yshift=-2.5cm] (1,0) node[right]{$H$} circle (1pt);
\filldraw[yshift=-2.5cm] (-1,0) node[right]{$H$} circle (1pt) ;
\filldraw[yshift=-2.5cm] (0,0) circle (.5pt) ;
\draw[thin,->,yshift=-2.5cm] (0,0) -- ($.8*(-100:1 and \yax)$) node[midway,left] {$u$};
\draw[thin,->,yshift=-2.5cm] (0,0) -- (.8,0) node[midway,above] {$x$};
\end{tikzpicture}
\caption{
The maps $\R \from[\snd] \ND[H]\SV\to[\pnd] \SV$ pictured in the case $\dim V=1$ and $N=\st0$. In this case, $H=\st{[0,1]}\sqcup\st{[0,-1]}$.
The thick lines in $\ND[H]\SV$ represent $\pnd^{-1}(H)$.
The red lines represent the fibers $\pnd^{-1}(N^+)$ and  $\pnd^{-1}(b)$, for $b\in V\setminus N$.
In the figure, we write $c_1=c\times\st1$ for $c\in \SV$.
}
\end{figure}

\begin{proof}
(i) Set $Z=(u=0)\subset W$. Then $\ND[Z]W = W\times\R = \R\times V\times\R$ with $\pnd^W(\tu,\tx,s) = (s\tu,\tx)$ and $\snd^W(\tu,\tx,s) = s$. The $\Gm$-action on $\ND[Z]W$ is given by $c\cdot(\tu,\tx,s)=(c\tu,\tx,c^{-1}s)$.

One has $\ND[\bdot Z]{\bdot W} = \opb{(\pnd^W)}(\bdot W) = \ND[Z]W \setminus \bl(\R\times N\times\st0)\union(\st0\times N\times\R)\br$.
Consider the $\Gmp$-action on $\ND[\bdot Z]{\bdot W}$ induced by the $\Gmp$-action on $\bdot W$, which is given by $c\cdot(\tu,\tx,s)=(c\tu,c\tx,s)$. Its quotient is the map $q\colon \ND[\bdot Z]{\bdot W} \to (\SV\times\R)\setminus(N_0^+\union N_0^-)$, given by $q(\tu,\tx,s)=([\tu,\tx],s)$. 
Setting  $\tilde p([\tu,\tx],s) = [s\tu,\tx]$ and $\tilde s([\tu,\tx],s) = s$, there is a commutative diagram with cartesian square
\[
\xymatrix@R=3ex{
\ND[\bdot Z]{\bdot W} \ar[r]_-q \ar[d]^-{\pnd^W} \ar@/^3ex/[rr]^{\snd^W} & (\SV\times\R)\setminus(N_0^+\union N_0^-) \ar[d]^-{\tilde p} \ar[r]_-{\tilde s} & \R \\
\bdot W \ar[r]^-q \ar@{}[ur]|\square & \SV.
}
\]
Since the quotient maps $q$ are principal $\Gmp$-bundles, it follows that
$\ND[H] \SV = (\SV\times\R)\setminus(N_0^+\union N_0^-)$, $\pnd=\tilde p$ and $\snd=\tilde s$.

\smallskip\noindent
(ii) follows by uniqueness of bordered compactifications.
\end{proof}

Denote by $T^+_H\SV \subset \bdot T_H\SV$  the normal vectors pointing to $V^+$.
Since
\[
T_H\SV = \snd^{-1}(0) = \bl\bdot V^+ \sqcup H \sqcup \bdot V^-\br\times\st0,
\]
this gives a natural identification $T^+_H\SV = \bdot V^+$.
We will also use the identification $V=V^+$ given by $\iota^+$.
Note that
the $\R^\times_{>0}$-action on $\bdot V^+\subset T_H\SV $
induced by the one on $T_H\SV$ is given by
$c\cdot\iota^+(x)= \iota^+(c^{-1}x)$.

By Lemma~\ref{le:EnuMU}, with the above identifications, $\Enu_H$ induces a functor
\[
\Enu_{H|\bdot V} \colon
\BEC \dVi \to \BEC \dVi.
\]
Similarly, $\Esm_V$ induces a functor (see \eqref{eq:smdot} below)
\[
\Esm_{V|\bdot V} \colon \BEC \dVi \to \BEC \dVi.
\]

\begin{lemma}\label{lem:nusm}
With the above notations, one has
\[
\Enu_{H|\bdot V} \simeq \Esm_{V|\bdot V}.
\]
\end{lemma}

\begin{proof}
Consider the diagram with cartesian squares
\[
\xymatrix@R=3ex{
\SV & \ar[l]_-\pnd \inb\Omega \ar[r]^-{j_\nd} & \inb{(\ND[H]\SV)} & \inb{(T_H\SV)} \ar[l]_-{i_\nd} \\
\dVi \ar@{^(->}[u]^{j_\SV} \ar@{_(->}[d]_{j_V} \ar@{}[ur]|-\square & \ar@<.5ex>[l]^-{\psm'} \ar@<-.5ex>[l]_-{\pnd'} \dVi\times\inb{(\R_{>0})} \ar[r]^-{j} \ar@{^(->}[u] \ar@{_(->}[d] \ar@{}[ur]|-\square & \dVi\times\inb\R \ar@{^(->}[u] \ar@{_(->}[d] \ar@{}[ur]|-\square & \dVi\rlap{$\ =\inb{(T_H^+\SV)}$} \ar[l]_-{ i} \ar@{^(->}[u] \ar@{_(->}[d] \ar@{^(->}[u]^{j'_\SV} \ar@{_(->}[d]_{j'_V} \\
\Vi \ar@{^(->}@/^2.0pc/[uu]^-j \ar@{}[ur]|-\square & \ar[l]_-\psm \Vi\times\inb{(\R_{>0})} \ar[r]^-{j_\sm} \ar@{}[ur]|-\square & \Vi\times\inb\R \ar@{}[ur]|-\square & \Vi \ar[l]_-{i_\sm}.
}
\]
Here, $\psm'$ and $\pnd'$ are induced by $\psm$ and $\pnd$, respectively.

Since
\begin{align*}
j_\SV(\pnd'(x,s)) &= \pnd([1,x],s) = [s,x] = [1,s^{-1}x], \\
j(j_V(p_\sm'(x,s))) &= j(\psm(x,s)) = j(s^{-1}x) = [1,s^{-1}x],
\end{align*}
one has in fact $\pnd'=\psm'$.

By \eqref{eq:EnuU}, one has
\[
\Enu_{H|\bdot V} \simeq \Eopb{ i} \circ \Eoim{{j}} \circ \Eopb{(\pnd')} .
\]
Similar arguments give
\begin{equation}\label{eq:smdot}
\Esm_{V|\bdot V}\simeq \Eopb{ i} \circ \Eoim{{j}} \circ \Eopb{(\psm')} .
\end{equation}
\end{proof}

\subsection{Smash functor and microlocalization}

As in the previous section, let $\tau\colon V\to N$ be a vector bundle.
Denote by $o\colon N\to V$ the embedding of the zero section.
Consider the natural identifications
\[
N=o(N)\subset V, \qquad T^*_NV =  V^*.
\]
There is the following relation between the smash functor and the Fourier-Sato transform

\begin{proposition}\label{pro:musm}
For $K\in\BECp\Vi$ there is an isomorphism in $\BECp\Wi$
\[
\Emu_N(K) \simeq \Esm_\W(\lap K).
\]
In other words, one has $\lap(\cdot)\circ\Enu_N\simeq\Esm_\W\circ\lap(\cdot)$.
\end{proposition}

\begin{proof}
Consider the following diagram with cartesian squares. 
\[
\xymatrix@R=4ex@C=7ex{
\Vi & \Vi\times_N\Wi \ar[l]^-p \ar[r]_-q & \Wi \\
\Vi\times\inb{(\R_{>0})} \ar[u]^{\pnd} \ar@{_(->}[d]_{ j_V} \ar@{}[ur]|-\square &(\Vi\times_N\Wi)\times\inb{(\R_{>0})} \ar[l]^-p \ar[r]_-q \ar[u]^{\pnd\times \psm} \ar@{_(->}[d]_ j \ar@{}[ur]|-\square &\ake \Wi\times\inb{(\R_{>0})} \ar[u]^{\psm} \ar@{_(->}[d]_{ j_{\W}} \\
\Vi\times\inb\R \ar@{}[ur]|-\square & (\Vi\times_N\Wi)\times\inb\R \ar[l]^-p \ar[r]_-q \ar@{}[ur]|-\square & \Wi\times\inb\R \\
\ake[2.5ex]\Vi \ar@{^(->}[u]^{i_V} \ar@{}[ur]|-\square & \Vi\times_N\Wi \ar[l]^-p \ar[r]_-q \ar@{^(->}[u]^k \ar@{}[ur]|-\square & \ake[2.5ex]\Wi \ar@{^(->}[u]^{i_{\W}}\,.
}
\]
We have to prove
\begin{equation}\label{eq:ndLsm}
\lap\bl\Eopb  i_V \Eoim {{j_V}} \Eopb p_\nd K\br
\simeq \Eopb  i_{\W} \Eoim {{j_{\W}}} \Eopb \psm \bl\lap K\br.
\end{equation}

Consider the left and right columns as morphisms of vector bundles over $N$, $N\times\R_{> 0}$, $N\times\R$ and $N$, respectively. Then the left column is the dual of the right column, and the middle column is the vector bundle product of the left and the right columns.
Moreover, since $(\prb\times \psm)(x,y,s) = (sx,s^{-1}y)$ and $\langle sx,s^{-1}y \rangle = \langle x,y \rangle$, 
the maps from the second row to the top row
is compatible with the coupling of 
the left columns and the right columns.

Therefore \eqref{eq:ndLsm} follows from Lemma~\ref{lem:lapinv}.
\end{proof}

\end{document}